\documentclass[11pt, oneside]{article}
\usepackage{amsmath, enumerate, amsthm, amssymb, wasysym, verbatim, bbm, color, graphics, geometry,graphicx,mathrsfs,}
\usepackage{algpseudocode,algorithm}
\usepackage{listings}
\usepackage{subcaption}
\usepackage{tikz}
\usepackage{mathdots,mathtools,enumitem}
\usepackage{yhmath}
\usepackage{cancel}
\usepackage{color}
\usepackage{siunitx}
\usepackage{array}
\usepackage{multirow}
\usepackage{physics}
\usepackage{yhmath}
\usepackage{gensymb}
\usepackage{tabularx}
\usepackage{extarrows}
\usepackage{booktabs}
\usetikzlibrary{fadings}
\usetikzlibrary{patterns}
\usetikzlibrary{shadows.blur}
\usetikzlibrary{shapes}

\tikzset{every picture/.style={line width=0.75pt}} 

\usepackage{float}

\usepackage{extarrows}
\usepackage{framed,bbm}
\usepackage[citecolor=blue,colorlinks=true,linkcolor=black]{hyperref}
\makeatletter
\newenvironment{breakablealgorithm}
  {
   \begin{center}
     \refstepcounter{algorithm}
     \hrule height.8pt depth0pt \kern2pt
     \renewcommand{\caption}[2][\relax]{
       {\raggedright\textbf{\fname@algorithm~\thealgorithm} ##2\par}%
       \ifx\relax##1\relax 
         \addcontentsline{loa}{algorithm}{\protect\numberline{\thealgorithm}##2}%
       \else 
         \addcontentsline{loa}{algorithm}{\protect\numberline{\thealgorithm}##1}%
       \fi
       \kern2pt\hrule\kern2pt
     }
  }{
     \kern2pt\hrule\relax
   \end{center}
  }
\makeatother

\newtheorem{thm}{Theorem}[section]

\newtheorem{lem}[thm]{Lemma}
\newtheorem{cor}[thm]{Corollary}

\newtheorem{prop}[thm]{Proposition}
\newtheorem*{asm*}{Assumption}

\theoremstyle{remark}
\newtheorem*{rem}{Remark}
\newtheorem{remark}{Remark}[section]

\newcommand{\md}{\mathrm{d}}
\newcommand{\me}{\mathrm{e}}
\newcommand{\eqd}{\overset{d}{=}}
\newcommand{\N}{\mathbb{N}}
\newcommand{\E}{\mathbb{E}}
\newcommand{\p}{\mathbb{P}}
\newcommand{\R}{\mathbb{R}}
\newcommand{\Z}{\mathbb{Z}}

\newcommand{\Unif}{\mathrm{U}}

\newcommand{\Exp}{\mathrm{Exp}}
\newcommand{\les}{\leqslant}
\newcommand{\ges}{\geqslant}

\numberwithin{equation}{section}
\numberwithin{figure}{section}

\newcommand{\1}{\mathbbm{1}}
\newcommand{\RP}{\mathbb{R}_+}
\newcommand{\Oh}{\mathcal{O}}

\begin{document}
\title{
Fast exact simulation of the first-passage event of a subordinator
}
\author{
Jorge Ignacio Gonz\'alez C\'azares\footnote{University of Warwick \& The Alan Turing Institute (\texttt{jorge.i.gonzalez-cazares@warwick.ac.uk})}, Feng Lin\footnote{University of Warwick  (\texttt{feng.lin.1@warwick.ac.uk})}, 
Aleksandar Mijatovi\'c\footnote{University of Warwick \& The Alan Turing Institute (\texttt{a.mijatovic@warwick.ac.uk})}}

\date{}
\maketitle
\begin{abstract}
 This paper provides an exact simulation algorithm for the sampling from the joint law of the first-passage time, the undershoot and the overshoot of a subordinator crossing a non-increasing boundary. We prove that the running time of this algorithm has finite moments of all positive orders and give an explicit bound on the expected running time in terms of the L\'evy measure of the subordinator. This bound provides performance guarantees that make our algorithm suitable for Monte Carlo estimation.  We provide a GitHub repository~\cite{repository} with an implementation of the algorithm in Python and Julia and a short \href{https://youtu.be/yQJY-fDBS6Q}{YouTube} presentation~\cite{Presentation_Aleks} describing the results.
\end{abstract}



\section{Introduction}\label{section:intro}
Simulation of the first-passage event of a L\'evy process across a barrier is of long standing interest in applied probability. For instance, first-passage events describe the steady-state waiting time and workload in queuing theory~\cite{MR1978607} and 
the ruin event (i.e. the ruin time and penalty function) in risk theory~\cite[Ch.~12]{doi:10.1142/7431}.
First-passage events also arise in financial mathematics in, for example, hedging and pricing barrier-type securities~\cite{MR3015232}. In all these application areas,  Monte Carlo methods for the first-passage event are  widely used~(see e.g.~\cite[Ex.~5.15]{MR2331321}).

Our main motivation for developing a fast exact simulation algorithm for the first-passage event of a subordinator comes from the role the event plays in the probabilistic representation of the solutions to a fractional partial differential equation (FPDE) (see e.g.~\cite{hernandez2017generalised} and the references therein, see also Subsection~\ref{subsec:Example_FPDE} below for a specific example). 
For this Feynman-Kac formula
to be of use in a Monte Carlo approximation of the solution of an FPDE, 
an exact simulation algorithm for the first-passage time of the subordinator is required.
In this context the fractional derivative in the FPDE corresponds to the generator of the subordinator and exact simulation is indeed required because approximate simulation (e.g. via a continuous-time random walk approximation~\cite{MR2791154})
may lead to errors that are hard to control, see the discussion in~\cite[Sec.~3.3.1]{cazares2023fast}.
In~\cite{MR4219829}, exact simulation of the first-passage time of a stable subordinator was used to solve numerically an FPDE with a Caputo fractional derivative~\cite{MR2854867}. Extending~\cite{MR4219829} beyond the stable case to general non-local derivatives requires an exact simulation algorithm for the first-passage event of general subordinators, along with an \textit{a priori} bound on its expected running time. Such an algorithm and explicit bounds on its expected running time are the main contributions of the present paper.

Our main algorithm (Algorithm~\ref{alg:triplet_Z} below) uses in an essential way the exact simulation algorithm for the first-passage event of a tempered stable subordinator recently developed in~\cite{cazares2023fast}, but also requires new ideas, both in simulation as well as in the analysis of the computational complexity of recursive simulation algorithms, such as  Algorithm~\ref{alg:triplet_Z}.  We give an overview of our algorithm and its complexity in Subsections~\ref{subsec:time_complexity} and~\ref{subsec:complexity_intro}, respectively. We  discuss the  exiting literature in the context of Algorithm~\ref{alg:triplet_Z} in Subsection~\ref{subsec:infinite_complexity_Chi} below. The key feature of Algorithm~\ref{alg:triplet_Z}, which makes it suitable for applications in Monte Carlo approximation of the solution of FPDE, is that its random running time has finite moments of all orders  with explicit  control over the expected running time in terms of model parameters (see Corollary~\ref{cor:time_complexity} below). Its application in Monte Carlo approximation  is a topic for future research. In this paper we discuss briefly how to use our algorithm in a numerical example, see Section~\ref{sec:applications} below. An implementation of Algorithm~\ref{alg:triplet_Z} in Python and Julia are available in the GitHub repository~\cite{repository}. A short 
\href{https://youtu.be/yQJY-fDBS6Q}{YouTube} presentation~\cite{Presentation_Aleks}
describs the algorithm, elements of the proof of our main complexity result (Theorem~\ref{thm:main_complexity} below) and numerical examples.


\subsection{Structure of Algorithm~\ref{alg:triplet_Z}}
\label{subsec:time_complexity}
Let $Z$
denote a pure-jump  driftless infinite activity subordinator,
started at $0$, with L\'evy measure $\nu_Z$ (see e.g.~\cite{ken1999levy} for background on L\'evy processes, including subordinators) admitting the representation
\begin{equation}
\label{eq:Levy_measure_description}
\nu_Z
=\nu_{r,q} + \lambda_r,\qquad \text{where}\quad\nu_{r,q}(\md x):=\1_{\{0<x\leq r\}}\vartheta\me^{-qx} x^{-\alpha-1}\md x, \quad x\in(0,\infty),
\end{equation}
for some parameters $\vartheta\in(0,\infty)$, $q\ges0$, $\alpha\in(0,1)$, $r\in(0,\infty]$ and a L\'evy measure $\lambda_r$ on $(0,\infty)$ with finite total mass $\Lambda_r:=\lambda_r(0,\infty)\in [0,\infty)$. 
The representation in~\eqref{eq:Levy_measure_description} 
requires only that the intensity of the small jumps of $Z$ are proportional to those of a tempered stable subordinator.\footnote{See definition of a tempered stable subordinator via its characteristic exponent in~\eqref{eq:def_tempered_stable} below.} Such a process $Z$ indeed belongs to 
a large class of subordinators:  if a L\'evy measure
$\nu_Z(\md x)=\nu_Z(x) \md x$ has 
a density $\nu_Z(x)$  near~$0$, satisfying
$\vartheta\coloneqq\lim_{x\downarrow0}\nu_Z(x)x^{\alpha+1}\in(0,\infty)$ and $\limsup_{x\downarrow0}|\nu_Z(x)x^{\alpha+1}-\vartheta|/x<\infty$, then~\eqref{eq:Levy_measure_description}
holds, see Lemma~\ref{lem:decompose} below.

In this paper we develop and analyse an exact simulation algorithm (see Algorithm~\ref{alg:triplet_Z} and its flowchart in Figure~\ref{fig:flowchart} below) for the random element 
\begin{equation}
    \label{eq:first_passage_event_vector}
\chi_c^Z:=\big(\tau_c^Z,
    Z_{\tau_c^Z-},
    Z_{\tau_c^Z}\big),
\qquad
\text{where}\quad 
\tau_c^Z\coloneqq\inf\{t>0\,:\,Z_t>c(t)\}
\end{equation}
and $c:[0,\infty)\to[0,\infty)$ is 
a non-increasing absolutely continuous function 
 satisfying $c_0:=c(0)\in(0,\infty)$.
The vector $\chi_c^Z$ represents the first-passage event of $Z$ over the boundary function $c$.
Since $\vartheta, r>0$, $c_0>0$ and the process $Z$ has a.s. increasing paths, we have $0<\tau_c^Z<\infty$. The left limits $Z_{t-}:=\lim_{s\uparrow t}Z(s)$ exist at all times $t>0$ a.s., with $Z_{\tau_c^Z-}$ denoting the level just before the crossing the boundary $c$.
Note that the pure-jump assumption does not restrict the scope of  Algorithm~\ref{alg:triplet_Z}: for a subordinator with drift $\mu>0$, we can apply 
Algorithm~\ref{alg:triplet_Z} to the driftless subordinator $\tilde Z_t:=Z_t-\mu t$ and the decreasing  boundary function $\tilde c(t):=c(t)-\mu t$, noting that $\chi_c^Z=\chi_{\tilde c}^{\tilde Z}+(0,\mu\tau_{\tilde c}^{\tilde Z},\mu\tau_{\tilde c}^{\tilde Z})$.

\begin{figure}
\centering
\begin{tikzpicture}[x=0.75pt,y=0.75pt,yscale=-1,xscale=1, scale=0.8, every node/.style={scale=0.8}]

\draw   (29,27) .. controls (29,15.4) and (105.11,6) .. (199,6) .. controls (292.89,6) and (369,15.4) .. (369,27) .. controls (369,38.6) and (292.89,48) .. (199,48) .. controls (105.11,48) and (29,38.6) .. (29,27) -- cycle ;
\draw   (82,70) -- (316,70) -- (316,110) -- (82,110) -- cycle ;
\draw   (43,140) -- (406,140) -- (406,180) -- (43,180) -- cycle ;
\draw   (197,213.5) -- (251,240) -- (197,266.5) -- (143,240) -- cycle ;
\draw   (361,223) -- (643,223) -- (643,263) -- (361,263) -- cycle ;
\draw   (362,302) -- (656,302) -- (656,342) -- (362,342) -- cycle ;
\draw   (59,295.5) -- (350,295.5) -- (350,348) -- (59,348) -- cycle ;
\draw   (91,381) -- (319,381) -- (319,421) -- (91,421) -- cycle ;
\draw   (198,453) -- (281,488) -- (198,523) -- (115,488) -- cycle ;
\draw   (494,452.5) -- (577,487.5) -- (494,522.5) -- (411,487.5) -- cycle ;
\draw    (199,48) -- (199,69.5) ;
\draw [shift={(199,71.5)}, rotate = 270] [color={rgb, 255:red, 0; green, 0; blue, 0 }  ][line width=0.75]    (10.93,-3.29) .. controls (6.95,-1.4) and (3.31,-0.3) .. (0,0) .. controls (3.31,0.3) and (6.95,1.4) .. (10.93,3.29)   ;
\draw    (200,110) -- (200,134.5) ;
\draw [shift={(200,136.5)}, rotate = 270] [color={rgb, 255:red, 0; green, 0; blue, 0 }  ][line width=0.75]    (10.93,-3.29) .. controls (6.95,-1.4) and (3.31,-0.3) .. (0,0) .. controls (3.31,0.3) and (6.95,1.4) .. (10.93,3.29)   ;
\draw    (198,181) -- (197.06,211.5) ;
\draw [shift={(197,213.5)}, rotate = 271.76] [color={rgb, 255:red, 0; green, 0; blue, 0 }  ][line width=0.75]    (10.93,-3.29) .. controls (6.95,-1.4) and (3.31,-0.3) .. (0,0) .. controls (3.31,0.3) and (6.95,1.4) .. (10.93,3.29)   ;
\draw    (197,266.5) -- (197,291) ;
\draw [shift={(197,293)}, rotate = 270] [color={rgb, 255:red, 0; green, 0; blue, 0 }  ][line width=0.75]    (10.93,-3.29) .. controls (6.95,-1.4) and (3.31,-0.3) .. (0,0) .. controls (3.31,0.3) and (6.95,1.4) .. (10.93,3.29)   ;
\draw    (197,348.5) -- (197.94,380.5) ;
\draw [shift={(198,382.5)}, rotate = 268.32] [color={rgb, 255:red, 0; green, 0; blue, 0 }  ][line width=0.75]    (10.93,-3.29) .. controls (6.95,-1.4) and (3.31,-0.3) .. (0,0) .. controls (3.31,0.3) and (6.95,1.4) .. (10.93,3.29)   ;
\draw    (251,240) -- (358,239.51) ;
\draw [shift={(360,239.5)}, rotate = 179.74] [color={rgb, 255:red, 0; green, 0; blue, 0 }  ][line width=0.75]    (10.93,-3.29) .. controls (6.95,-1.4) and (3.31,-0.3) .. (0,0) .. controls (3.31,0.3) and (6.95,1.4) .. (10.93,3.29)   ;
\draw    (495.5,263.25) -- (495.65,290.42) -- (495.53,299.5) ;
\draw [shift={(495.5,301.5)}, rotate = 270.78] [color={rgb, 255:red, 0; green, 0; blue, 0 }  ][line width=0.75]    (10.93,-3.29) .. controls (6.95,-1.4) and (3.31,-0.3) .. (0,0) .. controls (3.31,0.3) and (6.95,1.4) .. (10.93,3.29)   ;
\draw    (198,421.5) -- (198,451) ;
\draw [shift={(198,453)}, rotate = 270] [color={rgb, 255:red, 0; green, 0; blue, 0 }  ][line width=0.75]    (10.93,-3.29) .. controls (6.95,-1.4) and (3.31,-0.3) .. (0,0) .. controls (3.31,0.3) and (6.95,1.4) .. (10.93,3.29)   ;
\draw    (494.5,418.75) -- (494.03,450.5) ;
\draw [shift={(494,452.5)}, rotate = 270.85] [color={rgb, 255:red, 0; green, 0; blue, 0 }  ][line width=0.75]    (10.93,-3.29) .. controls (6.95,-1.4) and (3.31,-0.3) .. (0,0) .. controls (3.31,0.3) and (6.95,1.4) .. (10.93,3.29)   ;
\draw   (383,378) -- (611,378) -- (611,418) -- (383,418) -- cycle ;
\draw    (494.5,342.25) -- (494.97,376.5) ;
\draw [shift={(495,378.5)}, rotate = 269.21] [color={rgb, 255:red, 0; green, 0; blue, 0 }  ][line width=0.75]    (10.93,-3.29) .. controls (6.95,-1.4) and (3.31,-0.3) .. (0,0) .. controls (3.31,0.3) and (6.95,1.4) .. (10.93,3.29)   ;
\draw   (259,628) .. controls (259,616.95) and (297.73,608) .. (345.5,608) .. controls (393.27,608) and (432,616.95) .. (432,628) .. controls (432,639.05) and (393.27,648) .. (345.5,648) .. controls (297.73,648) and (259,639.05) .. (259,628) -- cycle ;
\draw    (345.5,579.25) -- (345.5,606) ;
\draw [shift={(345.5,608)}, rotate = 270] [color={rgb, 255:red, 0; green, 0; blue, 0 }  ][line width=0.75]    (10.93,-3.29) .. controls (6.95,-1.4) and (3.31,-0.3) .. (0,0) .. controls (3.31,0.3) and (6.95,1.4) .. (10.93,3.29)   ;
\draw    (198,523) -- (198,577.5) -- (374.5,579.25) ;
\draw    (374.5,579.25) -- (495.5,577.75) -- (494,522.5) ;
\draw    (577,487.5) -- (674,488.5) -- (670,91.5) -- (319,87.52) ;
\draw [shift={(317,87.5)}, rotate = 0.65] [color={rgb, 255:red, 0; green, 0; blue, 0 }  ][line width=0.75]    (10.93,-3.29) .. controls (6.95,-1.4) and (3.31,-0.3) .. (0,0) .. controls (3.31,0.3) and (6.95,1.4) .. (10.93,3.29)   ;
\draw    (115,488) -- (21,489.5) -- (19,162.5) -- (41,161.58) ;
\draw [shift={(43,161.5)}, rotate = 177.61] [color={rgb, 255:red, 0; green, 0; blue, 0 }  ][line width=0.75]    (10.93,-3.29) .. controls (6.95,-1.4) and (3.31,-0.3) .. (0,0) .. controls (3.31,0.3) and (6.95,1.4) .. (10.93,3.29)   ;

\draw (45,16) node [anchor=north west][inner sep=0.75pt]   [align=left] {$\displaystyle ( T,U,V)\leftarrow ( 0,0,0) ,\ b( t)\leftarrow \min\{r\rho ,c( t)\}$};
\draw (88,82) node [anchor=north west][inner sep=0.75pt]   [align=left] {Sample $\displaystyle ( D,J) \sim \text{Exp}( \Lambda_r ) \times \lambda_r /\Lambda_r $};
\draw (46,144) node [anchor=north west][inner sep=0.75pt]   [align=left] {Sample $\displaystyle ( T',U',V') \sim \mathcal{L}\left( \tau _{b}^{Y} ,Y_{\tau _{b}^{Y} -} ,Y_{\tau _{b}^{Y}}\right)$};
\draw (166,229) node [anchor=north west][inner sep=0.75pt]   [align=left] {$\displaystyle T'< D?$};
\draw (368,233) node [anchor=north west][inner sep=0.75pt]   [align=left] {Sample $\displaystyle W\sim \mathcal{L}( Y_D |\{Y_D < b( D)\})$};
\draw (365,312) node [anchor=north west][inner sep=0.75pt]   [align=left] {$\displaystyle ( T,U,V)\leftarrow ( T+D,V+W,V+W+J)$};
\draw (76,305) node [anchor=north west][inner sep=0.75pt]   [align=left] {$\displaystyle  \begin{array}{{>{\displaystyle}l}}
( T,U,V)\leftarrow ( T+T',V+U',V+V')\\
D\leftarrow D-T'
\end{array}$};
\draw (100,392) node [anchor=north west][inner sep=0.75pt]   [align=left] {$\displaystyle b( t)\leftarrow \min\{r\rho ,c( t+T) -V\}$};
\draw (159,477) node [anchor=north west][inner sep=0.75pt]   [align=left] {$\displaystyle b( 0) \les 0?$};
\draw (469,477) node [anchor=north west][inner sep=0.75pt]   [align=left] {$\displaystyle b( 0) \les 0?$};
\draw (398,389) node [anchor=north west][inner sep=0.75pt]   [align=left] {$\displaystyle b( t)\leftarrow \min\{r\rho ,c( t+T) -V\}$};
\draw (288,618) node [anchor=north west][inner sep=0.75pt]   [align=left] {Output $\displaystyle ( T,U,V)$};
\draw (205,269) node [anchor=north west][inner sep=0.75pt]   [align=left] {Yes};
\draw (211,540) node [anchor=north west][inner sep=0.75pt]   [align=left] {Yes};
\draw (506,544) node [anchor=north west][inner sep=0.75pt]   [align=left] {Yes};
\draw (640,460) node [anchor=north west][inner sep=0.75pt]   [align=left] {No};
\draw (71,466) node [anchor=north west][inner sep=0.75pt]   [align=left] {No};
\draw (300,224) node [anchor=north west][inner sep=0.75pt]   [align=left] {No};
\end{tikzpicture}
\caption{Flowchart of Algorithm~\ref{alg:triplet_Z} for the simulation of the random element $\big(\tau_c^Z,
    Z_{\tau_c^Z-},
    Z_{\tau_c^Z}\big)$, where the subordinator $Z$ is defined in~\eqref{eq:Levy_measure_description}. The pair $(D,J)$ is the first jump time and size of the compound Poisson process defined via the L\'evy measure $\lambda_r$ in~\eqref{eq:Levy_measure_description},  process $Y$ is the driftless truncated tempered stable subordinator with L\'evy measure $\nu_{r,q}$ given in~\eqref{eq:Levy_measure_description}, introduced at the beginning of Section~\ref{section:result}. There are two key steps in Algorithm~\ref{alg:triplet_Z}: (i) the simulation of the triplet $\left( \tau _{b}^{Y} ,Y_{\tau _{b}^{Y} -} ,Y_{\tau _{b}^{Y}}\right)$ (see line~\ref{step:fpe_truncated_tempered_stable} of Algorithm~\ref{alg:triplet_Z}), based on~\cite[Alg.~1]{cazares2023fast} for the tempered stable first-passage event, and (ii) the simulation from the conditional law $\mathcal{L}( Y_D |\{Y_D < b( D)\})$ via Algorithm~\ref{alg:small_tempered_stable} below, which is based on Devroye's  algorithm~\cite{devroye2012note} for log-concave densities. Algorithm~\ref{alg:small_tempered_stable} is critical for the bound on the expected running time in Corollary~\ref{cor:time_complexity} below, because $b(D)$ might be tiny with significant probability, making a naive simulation of $Y_D$, until  $Y_D<b(D)$, possibly have infinite expected running time. The parameter $\rho\in(0,1)$ can be optimised in terms of $\alpha$. Its optimal value lies in the interval $[1/\me,1/2]$ and we typically set it to $1/2$.}
    \label{fig:flowchart}
\end{figure}
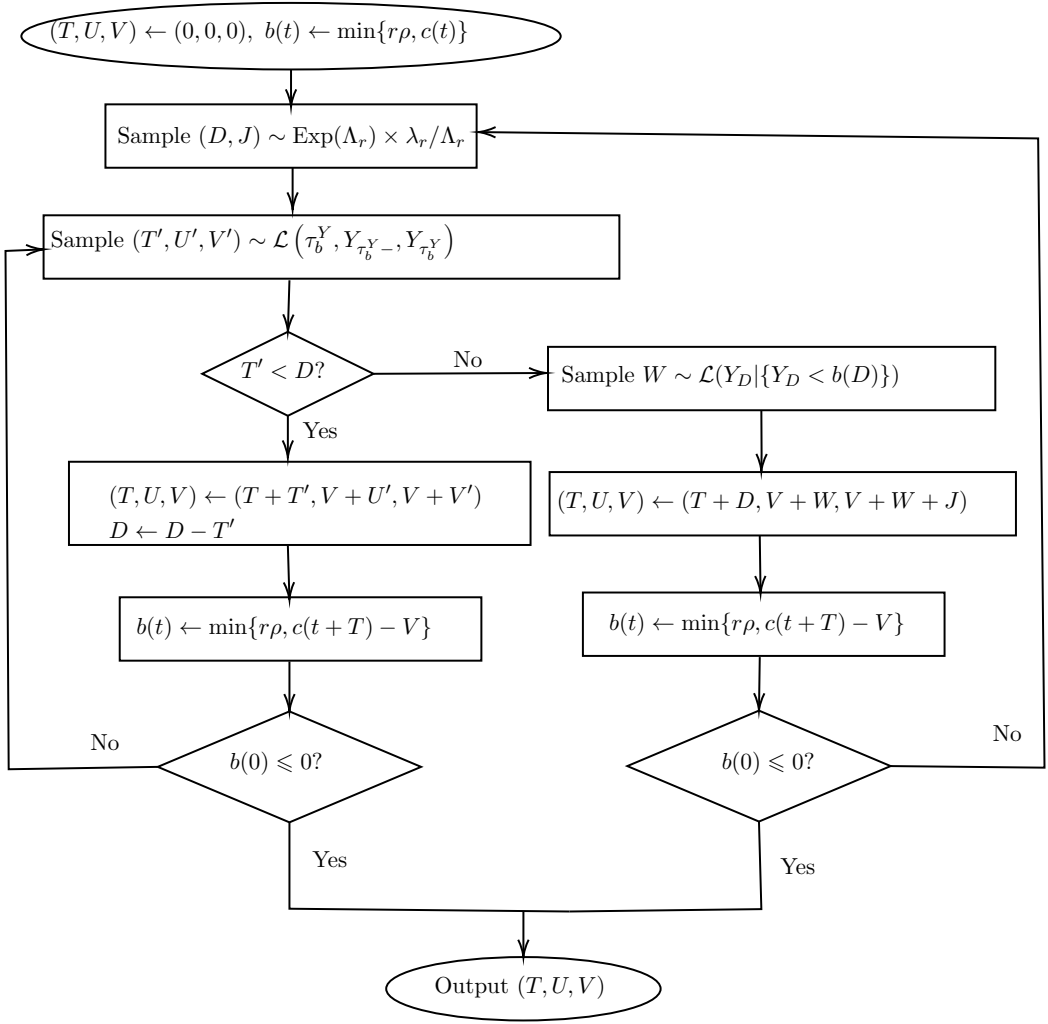



According to the flowchart in  Figure~\ref{fig:flowchart}, Algorithm~\ref{alg:triplet_Z} samples the triplet $ \chi_c^Z$ 
as follows: first it defines the barrier function $b(t)$, which is bounded above by the truncation level $r$ in~\eqref{eq:Levy_measure_description} (i.e.  $\rho\in(0,1)$) and the original function $c(t)$.
Then, Algorithm~\ref{alg:triplet_Z} samples 
the first jump time $D$, together with the corresponding  jump $J$, of the compound Poisson process with L\'evy measure $\lambda_r$ in~\eqref{eq:Levy_measure_description}. If this jump 
occurs after (resp. before) the truncated tempered stable process $Y$ (with L\'evy measure $\nu_{r,q}$ in~\eqref{eq:Levy_measure_description}) has crossed the barrier $b(t)$, 
the process $Z$ has crossed the barrier $b(t)$ (resp.
we sample $Y$, conditional on being below the barrier).
In both cases,
the barrier and the triplet are updated and the algorithm checks whether the original barrier $c(t)$ has been crossed. If not (resp. yes), the procedure is repeated on the updated barrier and starting point (resp. the updated triplet is returned).

The key steps in Algorithm~\ref{alg:triplet_Z} are as follows: (I) the simulation of the first-passage event of a truncated tempered stable subordinator, (II) the simulation of the increment of a truncated tempered stable subordinator, conditioned to be less than a given level. For (I),~\cite[Alg.~1]{cazares2023fast} provides an algorithm to sample the first-passage event of a tempered stable subordinator with running time that has an exponential moment. \cite[Alg.~1]{cazares2023fast} is applicable to the \textit{truncated} tempered stable process $Y$ 
because the target boundary satisfies $b(t)<r$, 
implying that only the jump at the crossing time can be greater than $r$, which Algorithm~\ref{alg:triplet_Y} corrects via thinning. 
 For (II), we apply Algorithm~\ref{alg:small_tempered_stable} below, based on the extremely efficient 
  Devroye's  algorithm~\cite{devroye2012note} for log-concave densities. As in~(I), Algorithm~\ref{alg:small_tempered_stable} for tempered stable processes is applicable to the \textit{truncated} tempered stable process $Y$, since 
   $b(D)<r$, where $r$ is the truncation level. 

\subsection{Complexity of Algorithm~\ref{alg:triplet_Z}}
\label{subsec:complexity_intro}

The expected computational complexity (or running time) of Algorithm~\ref{alg:triplet_Z} is given in Theorem~\ref{thm:main_complexity}
as a function of the parameters in~\eqref{eq:Levy_measure_description}
and the complexity of a simulation algorithm for the triplet in~\eqref{eq:first_passage_event_vector} when $Z$ is a tempered stable subordinator with L\'evy measure in~\eqref{eq:def_tempered_stable} below (note that $Z$ given in~\eqref{eq:Levy_measure_description} is tempered stable if and only if 
$r=\infty$ and $\Lambda_r=0$). 
The bound in Theorem~\ref{thm:main_complexity} suggests natural choices for the free parameters $\rho=1/2$ and $r=\max\{2\alpha/q,r_0\}$, see discussion in Subsection~\ref{subsec:Implementation} below. 

In Corollary~\ref{cor:time_complexity}, we provide an explicit upper bound on the expected running time of Algorithm~\ref{alg:triplet_Z} when the crucial tempered stable algorithm  for the simulation of the first-passage triplet is the one in our
recent paper~\cite{cazares2023fast}. 
 In the present paper we show that the random running time of Algorithm~\ref{alg:triplet_Z}  possesses moments of all orders and give an upper bound, which is (up to the cost of elementary operations such as evaluation of sums, products, as well as the elementary functions $\sin$, $\cos$, $\exp$ and the gamma function $\Gamma$) explicit in the model parameters, see Corollary~\ref{cor:time_complexity} below. In particular, this bound indicates the rate of the deterioration of the behaviour of the expected running time $\mathcal{T}(\alpha,\vartheta,q,c_0)$ of Algorithm~\ref{alg:triplet_Z} when model parameters in~\eqref{eq:Levy_measure_description} take extreme values\footnote{Recall that, by definition, we have $f(x)=\Oh(g(x))$ as $x\to x_*\in[0,\infty]$ if $\limsup_{x\to x_*}f(x)/g(x)<\infty$.}:
\begin{equation}
\label{eq:asymptotic-complexity}
\mathcal{T}(\alpha,\vartheta,q,c_0)=
\begin{cases}
    \Oh(|\log\alpha|\alpha^{-3}), 
        & \text{as }\alpha\to0,\\
    \Oh((1-\alpha)^{-3}),
        & \text{as }\alpha\to 1,\\
    \Oh(q(1+\1_{\{\Lambda_{r_0}=0\}}q^\alpha)),
        & \text{as }q\to\infty,
\end{cases}\quad 
\mathcal{T}(\alpha,\vartheta,q,c_0)=
\begin{cases}
    \Oh(c_0),
        & \text{as }c_0\to\infty,\\
    \Oh(1),
    &\text{as }\vartheta\to0,\\
    \Oh(1),
    &\text{as }\vartheta\to\infty.
\end{cases}
\end{equation}

In~\eqref{eq:asymptotic-complexity}, $r_0$ denotes 
the largest $r\in(0,\infty]$ for which the representation in~\eqref{eq:Levy_measure_description} holds (see paragraph of Eq.~\eqref{eq:decompostion_of_lambda} below for definition)
and $\Lambda_{r_0}=0$ corresponds to $Z$ being a 
truncated tempered stable subordinator if $r_0<\infty$ and tempered stable otherwise. By Corollary~\ref{cor:time_complexity} below, in the regimes $\alpha\to1$ or $c_0\to\infty$, Algorithm~\ref{alg:triplet_Z} for a general subordinator $Z$ given by~\eqref{eq:Levy_measure_description}
exhibits the same deterioration in performance as~\cite[Alg.~2]{cazares2023fast} for the tempered stable subordinator, see~\cite[Cor.~2.2]{cazares2023fast}. In the regime $\alpha\to1$, this is expected for the following reasons: (I) the acceptance probability in line~\ref{step:accept_jump<r} of Algorithm~\ref{alg:triplet_Y} (called in line~\ref{step:fpe_truncated_tempered_stable} of Algorithm~\ref{alg:triplet_Z}) converges to $1/2$, where the main work in Algorithm~\ref{alg:triplet_Y} is done by~\cite[Alg.~1]{cazares2023fast} for tempered stable subordinators,
and (II) the number of jumps of the compound Poisson process with jump measure $\lambda_r$ during the time interval $[0,\tau_c^Z]$ decreases since the infinite activity part of $Z$ increases, making $\tau_c^Z$ smaller. In contrast, as $\alpha\to0$ or $q\to\infty$, the performance of Algorithm~\ref{alg:triplet_Z} generally deteriorates more than~\cite[Alg.~2]{cazares2023fast} (recall from~\cite[Cor.~2.2]{cazares2023fast} that in the tempered stable case the bound is $\Oh(q)$ as $q\to\infty$ and $\Oh(|\log\alpha|\alpha^{-2})$ as $\alpha\to0$)\footnote{Note  that a difference in behaviour between the regimes $\alpha\to1$ and $\alpha\to0$ 
is analogous to the behaviour in~\cite{MR4122822}:
the expected running time of~\cite[Alg.~4.1]{MR4122822}, which samples $(\tau_c^S,S_{\tau_c^S})$ for $c\equiv r$ and a truncated tempered stable subordinator $S$, is bounded as $\alpha\to 1$ and is proportional to $1/\alpha$ as $\alpha\to 0$. These bounds are not explicitly stated in~\cite{MR4122822}, but can be easily obtained from the analysis in~\cite{MR4122822}.}. This occurs since, as $\alpha\to0$, the acceptance probability in line~\ref{step:accept_jump<r} of Algorithm~\ref{alg:triplet_Y} is asymptotically equivalent to $1-2^{-\alpha}$, and hence $\alpha\log2$, as $\alpha\to0$. Moreover, if $\Lambda_{r_0}=0$, then as $q\to\infty$ the expected number of jumps of the compound Poisson component with L\'evy measure $\lambda_r$ until time $\tau_c^Z$ grows as $q^{1+\alpha}$. Based on the numerical study  in~\cite[Subsec.~3.1]{cazares2023fast}, we expect that in the regime $q\to\infty$ the deterioration is close to the upper bound, while in the case $\alpha\to0$ and $\alpha\to1$,
the actual expected running time of Algorithm~\ref{alg:triplet_Z} grows more slowly than the bound in Corollary~\ref{cor:time_complexity} below suggests.

In Figure~\ref{fig:implematation_running_time_alpha} we present the observed expected complexity of the algorithm for the suggested value $r=\min\{r_0,2\alpha/q\}$ and for the range $\alpha\in[.05,.95]$. The  complexity as $\alpha\to 0$ exhibits explosion  as predicted by Theorem~\ref{thm:main_complexity}. However, as $\alpha\to 1$, the exhibited computational complexity appears to be much better behaved than predicted. Indeed, the observed complexity appears not to deteriorate as $\alpha\to 1$.

\begin{figure}
\centering
\includegraphics[width=.75\textwidth]{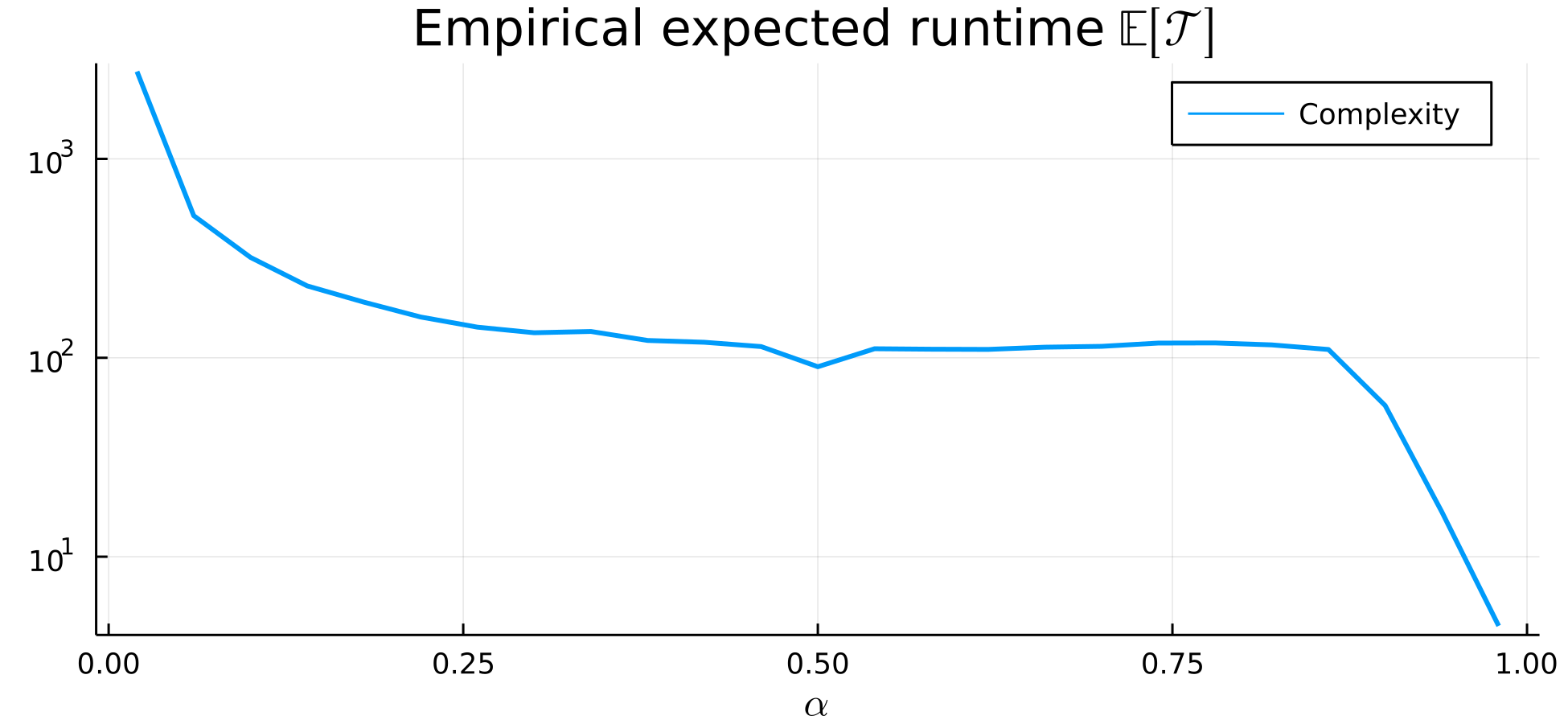}        
\caption{The graph plots the complexity of the implementation in~\cite{repository}  of Algorithm~\ref{alg:triplet_Z} with parameters $\vartheta=2$, $r_0=\infty$, $c(t)\equiv 5$, $\lambda_{r_0}(\md x)=\me^{-x}\md x$, $q=10$ and the truncation $r=\min\{r_0,2\alpha/q\}=2\alpha/q$ for the range $\alpha\in[.05,.95]$. The average running time $\E[\mathcal T]$ is measured in seconds taken for every $10^4$ samples.}
\label{fig:implematation_running_time_alpha}
\end{figure}

\subsection{Comparison with the literature}
\label{subsec:infinite_complexity_Chi}

As explained in the Subsection~\ref{subsec:time_complexity} above,~
\cite[Alg.~1]{cazares2023fast} plays a key role in the main algorithm 
of the present paper. 
However, the computational feasibility of Algorithm~\ref{alg:triplet_Z} rests on two additional ingredients:  efficient simulation of a tempered stable variable, conditioned on being small and sequential analysis of the complexity of an iterative algorithm, where the iterations are not independent of each other. The later is reflected in the proof of our main results, Theorem~\ref{thm:main_complexity} below, through the elementary but crucial Lemma~\ref{lem:comp_cost}. Neither of these ingredients feature in~\cite{cazares2023fast}.

To the best of our knowledge, the main algorithm in~\cite{chi2016exact} is the only other exact simulation procedure
applicable to the class of subordinators in~\eqref{eq:Levy_measure_description}, considered in this paper, and 
general absolutely continuous non-increasing boundary functions.
The main algorithm in~\cite{chi2016exact} 
and Algorithm~\ref{alg:triplet_Z}
use related ingredients, but their structure is different. 
As discussed in~\cite[Sec.~2.4]{cazares2023fast}, the main simulation algorithm in~\cite{chi2016exact} typically has  
an infinite expected running time (already in the case when $Z$ is a stable subordinator), making it unsuitable for Monte Carlo estimation.

In order to understand the relationship of Algorithm~\ref{alg:triplet_Z} and that in~\cite{chi2016exact}, we first recall the main idea and the key simulation steps in~\cite{chi2016exact}. In order to sample the first-passage triplet for $Z$,~\cite{chi2016exact} simulates the triplet for the truncated (at $r$) tempered stable process $Y$. The main idea in~\cite{chi2016exact} is to embed  $Y$ into the stable subordinator $\zeta$ as follows: express
$\zeta= S+P$, where $S$ is a tempered stable subordinator and  $P$ is an independent driftless compound Poisson process with L\'evy measure 
$\vartheta(1-\me^{-qx})x^{-\alpha-1}\md x$.
Note that, for any $q>0$,
the L\'evy measure of $P$ has finite mass and that its sum with the L\'evy measure $\nu_q$ of $S$
in~\eqref{eq:def_tempered_stable} yields a stable L\'evy measure $\vartheta x^{-\alpha-1}\md x$.
In order to sample
samples $Y_{\tau_{\bar b}^Y-}$ at the crossing time $\tau_{\bar b}^Y$
over the boundary function
boundary $\bar b(t):=r\wedge c(t)$,~\cite{chi2016exact} proceeds in two steps: first, it samples the undershoot of the stable process $s:=\zeta_{\tau_{\bar b}^\zeta-}$ and, second, the summands in
$\zeta_{\tau_{\bar b}^\zeta-}=S_{\tau_{\bar b}^\zeta-}+P_{\tau_{\bar b}^\zeta-}=s$, conditional on the value of the sum.
As shown in~\cite[Sec.~2.4]{cazares2023fast},
the simulation of the stable undershoot $\zeta_{\tau_{\bar b}^\zeta-}$ in
line~3 of~\cite[Sec.~7]{chi2016exact}
has infinite expected running time. 

It is remarkable that line~5 of~\cite[Sec.~7]{chi2016exact}  simulates the bridge $S_{\tau_{\bar b}^\zeta-}+P_{\tau_{\bar b}^\zeta-}=s$, given 
the lack of the analytical tractability of the densities of the summands (it is in fact not even clear which of the two summands  gives rise to the crossing of the stable subordinator $\zeta=S+P$). To put the achievement of~\cite{chi2016exact} in perspective, note that the main underlying technical problem solved in~\cite{MR3500619} is that of a simulation of a bridge, where the two summands are independent stable random variables.
Given the complexity of the task undertaken in~\cite[Sec.~7]{chi2016exact}, it is perhaps  not surprising that it will typically have  infinite expected running time, see Appendix~\ref{app:Chi} below for a proof of this assertion. 

The main conclusion of the above discussion is that one cannot simply apply the fast exact simulation algorithm   for a tempered stable subordinator in~\cite{cazares2023fast}, together with the main algorithm in~\cite{chi2016exact}, to obtain an efficient exact algorithm for the first-passage event in~\eqref{eq:first_passage_event_vector}
for a general subordinator $Z$.
We conclude this section by noting that~\cite[Alg.~4.1]{MR4122822} cannot be applied directly to the first-passage simulation problem of the truncated tempered stable subordinator considered in this paper for the following reasons:~\cite{MR4122822} does not sample the undershoot  and considers constant barrier functions only, as the main aim of~\cite{MR4122822} is to sample a marginal distribution of a truncated tempered stable process.

\subsection{Organization of this paper}
Our main algorithm and the results on the expected computational complexity are provided in Section~\ref{section:result}. Section~\ref{sec:alg_small_tempered_stable} describes the algorithm to simulate a tempered stable random variable conditioned to be small.  In Section~\ref{sec:applications} we 
discuss how to choose the truncation level in practice and 
present a
numerical applications of Algorithm~\ref{alg:triplet_Z} in Monte Carlo estimation of the solutions of FPDEs.  The proofs of the validity of Algorithm~\ref{alg:triplet_Z} and the analysis of its computational complexity are given in Section~\ref{sec:proof}. Appendices~\ref{sec:class_of_subordinators} and~\ref{sec:hitting_times}
provide some known results on the behaviour of the L\'evy density close to zero and exponential moments of crossing times, respectively. Appendix~\ref{app:Chi} analyses the expected running time of the main algorithm in~\cite{chi2016exact}.
 
\section{Sampling of the first passage event of a subordinator}\label{section:result}

Let the L\'evy measures 
$\lambda_r$
and $\nu_{r,q}$
be as in~\eqref{eq:Levy_measure_description} with parameters
$\alpha\in(0,1)$, $\vartheta,r\in(0,\infty)$, $q\in[0,\infty)$. 
Recall that $\Lambda_r=\lambda_r(0,\infty)<\infty$ and denote by 
$Q$ the compound Poisson process with L\'evy measure $\lambda_r$. Let $S$ be a driftless tempered stable subordinator, started at $0$, with Laplace transform and L\'evy measure 
\begin{equation}
\label{eq:def_tempered_stable}
\E[\me^{-uS_1}]
=\exp\big((q^\alpha-(u+q)^\alpha)\theta\big),
    \enskip u\ges0
\quad \&\quad \nu_q(\md x)
=(\alpha/\Gamma(1-\alpha))\theta\me^{-qx}x^{-\alpha-1}\md x,
    \enskip x>0,
\end{equation}
respectively. Here we set $\theta:=\vartheta \Gamma(1-\alpha)/\alpha\in(0,\infty)$, where $\Gamma$ is the Gamma function, so that the Laplace transform of $S_1$ has the form in~\eqref{eq:def_tempered_stable}. Let $Y$ be the subordinator equal to $S$ with all the jumps of size greater than $r$ removed. Then $Y$ is a driftless subordinator with L\'evy measure $\nu_{r,q}$, independent of $Q$, and the subordinator $Y+Q$ has the same law as $Z$. 

\begin{asm*}[FPTS-Alg]
\label{alg:FPTS}
Given a non-increasing absolutely continuous function $f:[0,\infty)\to[0,\infty)$ with $f(0)>0$. Assume that the running time of the first-passage tempered-stable algorithm (\nameref{alg:FPTS}) 
for generating a sample from the law of 
$\big(\tau_f^S,
    S_{\tau_f^S-},
    S_{\tau_f^S}\big)$ 
has a finite exponential moment that, for fixed $(\alpha,\vartheta,q)$ and $\mathcal{K}\in(0,\infty)$, is uniformly bounded for $f(0)\in [0,\mathcal{K}]$. Denote by $\mathcal{T}_0=\mathcal{T}_0(\alpha,\vartheta,q,\mathcal{K})$ a uniform upper bound on the expected running time of~\nameref{alg:FPTS} for $f(0)\in[0,\mathcal{K}]$.
\end{asm*}

A possible choice of~\nameref{alg:FPTS} is~\cite[Alg.~1]{cazares2023fast} 
with a bound $\mathcal{T}_0$ given in~\cite[Thm~2.1]{cazares2023fast}.
Denote by $\Exp(\Lambda_r)$ the exponential law with mean $1/\Lambda_r$. If $\Lambda_r=0$ (i.e. $Z$ equals a truncated tempered stable process), the corresponding random variable equals $\infty$ a.s. 
We can now state our main simulation algorithm.

\begin{breakablealgorithm}
\caption{Simulation of the triplet $(\tau_c^Z,Z_{\tau_c^Z-},Z_{\tau_c^Z})$ of a subordinator with L\'evy measure in~\eqref{eq:Levy_measure_description}}
\label{alg:triplet_Z}
\begin{algorithmic}[1]
\Require{Parameters $\alpha,\rho\in(0,1)$, $\vartheta\in(0,\infty)$, $r\in(0,\infty]$, $q\in[0,\infty)$, finite measure $\lambda_r$ in~\eqref{eq:Levy_measure_description} and non-increasing function $c$}
\State{Set $(T,U,V)\gets(0,0,0)$, $b(t)=\min\{c(t),r\rho\}$, $\theta=\vartheta \Gamma(1-\alpha)/\alpha$ and generate $D\sim\Exp(\Lambda_r)$}
\Repeat\label{step:fpe_truncated_tempered_stable}{~generate $(T',U',V')\sim \mathcal{L}\left(\tau_{b}^{Y} ,Y_{\tau_{b}^{Y}-}, Y_{\tau_{b}^{Y}}\right)$ via Algorithm~\ref{alg:triplet_Y}}
        \If{$D > T'$} \label{step:dichotomy}
        \Comment{$Y$ crosses $b$ before $Q$ jumps}
            \State{Set $(T,U,V)\gets(T+T', V+U', V+V')$} \label{step:main_4}
            \State{Set $D\gets D-T'$ and $b(t)=\min\{c(t+T)-V,r\rho\}$}
            \label{step:main_5}
         \ElsIf{$D \les T'$}
            \Comment{$Q$ jumps before $Y$ crosses $b$}
            \State{Generate $W\sim \mathcal{L}\left(Y_D|\{Y_D<b(D)\}\right)$ via Algorithm~\ref{alg:small_tempered_stable} and $J\sim\lambda_r/\Lambda_r$}\label{step:small_tempered_stable}
            \State{Set $(T,U,V)\gets(T+D, V+W, V+W+J)$}\label{step:main_14}
            \State{Generate $D\sim\Exp(\Lambda_r)$ and set $b(t)=\min\{c(t+T)-V,r\rho\}$}
        \EndIf
\Until{$b(0)\les 0$}
\State{\Return $(T,U,V)$}
\end{algorithmic}
\end{breakablealgorithm}

The key step in Algorithm~\ref{alg:triplet_Z} is the simulation of the first passage of the truncated tempered stable subordinator, given in the following algorithm.

\begin{breakablealgorithm}
\caption{Simulation of $(\tau_b^Y,Y_{\tau_b^Y-},Y_{\tau_b^Y})$ for a truncated tempered stable subordinator $Y$ and a function $b:[0,\infty)\to[0,\infty)$ bounded by the truncation level $r$}
\label{alg:triplet_Y}
\begin{algorithmic}[1]
\Require{Parameters $\alpha\in(0,1)$, $\vartheta\in(0,\infty)$, $r\in(0,\infty]$, $q\in[0,\infty)$ and $b\les r$}
\Repeat\label{step:fpe_truncated_tempered_stable_alg_2}
        {generate $(T,U,V)\sim \mathcal{L}\left(\tau_{b}^{S} ,S_{\tau_{b}^{S}-} ,S_{\tau_{b}^{S}}\right)$} via~\nameref{alg:FPTS}      \label{step:triplet_tempered_stable}
    \Until{$V-U\les r$}\label{step:accept_jump<r}
    \Comment{$T$ is the first crossing time of $Y$ over $b$}
\State{\Return $(T,U,V)$}
\end{algorithmic}
\end{breakablealgorithm}

The dichotomy in line~\ref{step:dichotomy} of Algorithm~\ref{alg:triplet_Z} is based on whether  the compound Poisson process $Q$ jumped before or after the truncated tempered stable subordinator $Y$ crossed a given boundary. The two events are depicted in Figure~\ref{fig:jump}. If $Y$ crosses first, then, since $Q$ has not jumped, $Z$ equals $Y$ at the time  $\tau_b^Y$ of crossing. The reason why we may sample in Algorithm~\ref{alg:triplet_Y}
the first-passage event of a tempered stable subordinator $S$ (without truncation at $r$) is because the trajectories of $Y$ and $S$ are equal \textit{before} the first crossing time $\tau_b^S$ since the function $b$ is bounded above by $r$
 

\begin{figure}[ht]
\centering
\includegraphics[scale=0.8]{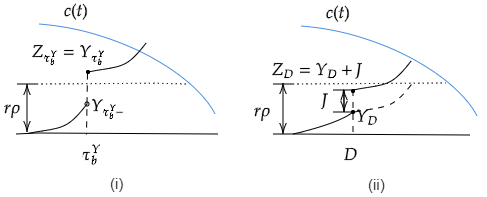}
\caption{In (i),  $Y$ (and hence $Z$) crosses the boundary $b(t):=\min\{c(t),r\rho\}$  before the first jump of $Q$ at time $D$. 
Note that in (ii), on the event $\{D<\tau_b^S\}$, we have $Y_D<b(D)$ but it is possible to have either $Y_D+J>b(D)$ or $Y_D+J\les b(D)$ ($J$ is the magnitude of the jump of $Q$ at time $D$).}
\label{fig:jump}
\end{figure}

The truncation parameter $r$ 
in the representation in~\eqref{eq:Levy_measure_description} of the  L\'evy measure $\nu_Z$
 is not uniquely determined. 
 Let $r_0\in(0,\infty]$ be supremum over all $r\in(0,\infty]$, for which the representation of the L\'evy measure $\nu_Z$ in~\eqref{eq:Levy_measure_description} is valid, i.e. $\lambda_r$ is a measure on $(0,\infty)$ with $\lambda_r(0,\infty)=\Lambda_r\in[0,\infty)$.
 Changing $r$ in~\eqref{eq:Levy_measure_description} clearly modifies the finite activity L\'evy measure $\lambda_r$. 
 Note that the representation in~\eqref{eq:Levy_measure_description} holds for the truncation level $r_0$ with the total mass of the compound  Poisson component $\Lambda_{r_0}\in[0,\infty)$ possibly equal to zero.
 Moreover, for any truncation level $r\in(0,r_0]$, the measure $\lambda_r$ in~\eqref{eq:Levy_measure_description} satisfies
\begin{equation}
\label{eq:decompostion_of_lambda}
\lambda_r(0,\infty)=\Lambda_r
=\Lambda_{r_0}+\int_r^{r_0}\vartheta \me^{-qx}x^{-\alpha-1}\md x.
\end{equation}

Define, for $r>0$ and $u,q\ges 0$, the function
\begin{equation}
\label{eq:upsilon_function}
\Upsilon(u,r,q)
\coloneqq\frac{1}{2}\begin{cases}
u(r\wedge 1),&u+q<r^{-1}\vee 1,\\
1-(r\wedge 1)q+\log((u+q)(r\wedge 1)),&q<r^{-1}\vee 1\les u+q,\\
\log(1+u/q),&r^{-1}\vee 1\les q.
\end{cases}
\end{equation}
Here and throughout the paper, $x\wedge y$ (resp. $x\vee y$) denotes $\min\{x,y\}$ (resp. $\max\{x,y\}$) for any $x,y\in\R$.

\begin{thm}
\label{thm:main_complexity}
Let $Z$
be a driftless subordinator with L\'evy measure given in~\eqref{eq:Levy_measure_description} and $c:[0,\infty)\to[0,\infty)$ 
as in Assumption~\nameref{alg:FPTS}. Then
Algorithm~\ref{alg:triplet_Z} samples the from the law of the random vector $\chi_c^Z$
in~\eqref{eq:first_passage_event_vector}. The running time of Algorithm~\ref{alg:triplet_Z} has moments of all order if~\nameref{alg:FPTS} does. The expected running time of Algorithm~\ref{alg:triplet_Z} is bounded above by
\begin{equation}
\label{eq:main_complexity}
\kappa_{\ref{alg:triplet_Z}}\left(\frac{\mathcal{T}_0(\alpha,\vartheta,q,\min\{c_0,r\rho\})}{1-\rho^\alpha}+\me^{q\min\{c_0,r\rho\}}\right)
\left(\frac{\Lambda_r}{\psi_0+\vartheta\Upsilon(1/(1+c_0),r_0,q)}+\frac{c_0}{r\rho}\right),
\end{equation}
where $\mathcal{T}_0$ denotes the expected running time of~\nameref{alg:FPTS}, $\psi_0:=\int_{(0,\infty)}(1-\me^{-x/(1+c_0)})\lambda_{r_0} (\md x)$ and the constant $\kappa_{\ref{alg:triplet_Z}}$ does not depend on $\alpha\in(0,1)$, $\vartheta\in(0,\infty)$, $c_0\coloneqq c(0)\in(0,\infty)$, $\Lambda_r\in(0,\infty)$, $\rho\in(0,1)$, $r\in(0,r_0]$ or $q\in[0,\infty)$.
\end{thm}

The proof of Theorem~\ref{thm:main_complexity} is given in Section~\ref{sec:proof} below. 
We assume in Theorem~\ref{thm:main_complexity} and throughout the paper that the complexity of the simulation of the jump from the law $\lambda_r/\Lambda_r$ does not depend on $r$ and the running time of this simulation has a finite exponential moment. Typically, simulation from $\lambda_r/\Lambda_r$ will have constant complexity (e.g. inversion of the distribution function).

Note that
$\psi_0>0$ if and only if $\Lambda_{r_0}>0$ (i.e. $Z$
is not truncated tempered stable), making the fraction
$1/(\psi_0+\vartheta\Upsilon(1/(1+c_0),r_0,q))$ 
 in~\eqref{eq:main_complexity} bounded away from zero and hence of constant order in any regime. 
In addition, $\Upsilon$ satisfies the following limits $\lim_{r\to0}\Upsilon(u,r,q)/r=u/2$, $\lim_{q\to\infty}\Upsilon(u,r,q)q= u/2$ and $\lim_{u\to 0}\Upsilon(u,r,q)/u=(q^{-1}\wedge r\wedge 1)/2$. In Theorem~\ref{thm:main_complexity}, the dependence in $\rho$ is via the coefficients $1/\rho$ and $1/(1-\rho^\alpha)$. This suggests that $\rho$ should be chosen away from $0$ and $1$. In particular, $\rho=1/2$ is a reasonable choice. 
In practice, a value different from $1/2$ might perform slightly better and should be chosen on a case-by-case basis, as a function of all other parameters.

\begin{cor}
\label{cor:time_complexity}
If \nameref{alg:FPTS} is given by~\cite[Alg.~1]{cazares2023fast}, with the choice $\rho=1/2$ and $r=\min\{2\alpha/q,r_0\}$, the expected running time $\mathcal{T}(\alpha,\vartheta,q,c_0)$ of Algorithm~\ref{alg:triplet_Z} is bounded above by 
\[
\kappa_2
\big((1-\alpha)^{-3}+|\log\alpha|+\log N\big)
\left(\frac{\alpha\Lambda_{r_0}
    +\vartheta\1_{\{2\alpha/q<r_0\}}q^\alpha}{\psi_0+\vartheta\Upsilon(1/(1+c_0),r_0,q)}
    +c_0q\right)/\alpha^3
\]
where the constant $\kappa_2$ does not depend  on any of the quantities: $\alpha\in(0,1)$, $\vartheta\in(0,\infty)$, $c_0=c(0)\in(0,\infty)$, $q\in(0,\infty)$ or $N\in\N$. Moreover, the asymptotic behaviour in~\eqref{eq:asymptotic-complexity} above holds.
\end{cor}

As in~\cite{cazares2023fast}, $N\in\N$ specifies the precision bits used by the computer (typical values of $N$ are $64$ or $128$). The choice of $r=\min\{2\alpha/q,r_0\}$ is discussed in Subsection~\ref{subsec:Implementation}.

\begin{proof}[Proof of Corollary~\ref{cor:time_complexity}]
Since $r\les 2\alpha/q$ and \nameref{alg:FPTS} (i.e.~\cite[Alg.~1]{cazares2023fast})
in Algorithm~\ref{alg:triplet_Z} is used for a boundary function bounded above by $r$, the bound on the expected running time $\mathcal{T}_0$ in~\cite[Thm~2.1]{cazares2023fast} implies
$$\mathcal{T}_0(\alpha,\vartheta,q,\min\{c_0,r\rho\})/\alpha\les \kappa'
\me^{4qr/\alpha}((1-\alpha)^{-3}+|\log\alpha|+\log N)/\alpha^2\les \kappa''((1-\alpha)^{-3}+|\log\alpha|+\log N)/\alpha^2 $$
for constants $\kappa',\kappa''$, not dependent on $\alpha\in(0,1)$, $\vartheta\in(0,\infty)$, $c_0\coloneqq c(0)\in(0,\infty)$, $\Lambda_r\in(0,\infty)$, $q\in(0,\infty)$ or $N\in\N$.
Note that, by~\eqref{eq:decompostion_of_lambda} and the inequality $(2\alpha)^{-\alpha}\les 2$, we have
\begin{equation}
\label{eq:lambda_r}
\Lambda_r
\les\Lambda_{r_0}+ \vartheta\me^{-qr}\int_r^{r_0}x^{-\alpha-1}\md x
\les\Lambda_{r_0}+\vartheta(r^{-\alpha}-r_0^{-\alpha})/\alpha
\les\Lambda_{r_0}+\vartheta\1_{\{2\alpha/q<r_0\}}2q^\alpha/\alpha.
\end{equation} 
Since $\rho=1/2$, the convexity of $1-1/2^\alpha$ yields  $\mathcal{T}_0/(1-\rho^\alpha)\les 2\mathcal{T}_0 /\alpha$.
Theorem~\ref{thm:main_complexity} implies that the expected complexity of Algorithm~\ref{alg:triplet_Z} is bounded above by
\[
\kappa_1'
((1-\alpha)^{-3}+|\log\alpha|+\log N)
\left(\frac{\Lambda_{r_0}+\vartheta\1_{\{2\alpha/q<r_0\}}q^\alpha/\alpha}{\psi_0+\vartheta\Upsilon(1/(1+c_0),r_0,q)}
+c_0q/\alpha\right)/\alpha^2,
\]
where $\kappa_{\ref{alg:triplet_Z}}'=4\me\kappa_{\ref{alg:triplet_Z}}\kappa''$. 
The bound in the corollary follows, which in turn implies the asymptotic behaviour stated in~\eqref{eq:asymptotic-complexity}. 
\end{proof}

\section{Sampling tempered stable random variables conditioned to be small}
\label{sec:alg_small_tempered_stable}

Line~\ref{step:small_tempered_stable} in Algorithm~\ref{alg:triplet_Z} requires an efficient way of sampling a tempered stable law conditioned to be small. We first present an algorithm for sampling a stable law (i.e. the law in~\eqref{eq:def_tempered_stable} with $q=0$) conditioned to be small. 
The stable density has a well known Zolotarev integral representation, see~\eqref{eq:zolotarev_representation} below, featuring the function 
\begin{equation}
\label{eq:zolotarev_density}
\sigma_\alpha(u)
\coloneqq 
\left(\frac{\sin(\alpha\pi u)^{\alpha}\sin((1-\alpha)\pi u)^{1-\alpha}}{\sin(\pi u)}\right)^{\beta+1},
\quad u\in(0,1),
\quad\text{where  $\beta:=\alpha/(1-\alpha)\in(0,\infty)$.}
\end{equation}
It is not hard to prove, see e.g.~\cite[Lem.~4.3]{cazares2023fast}, that $\sigma_\alpha$ is convex, 
making it possible to 
use 
Devroye's  algorithm~\cite{devroye2012note}
 for log-concave densities within Zolotarev's representation. (Note that~\cite[Alg.~11]{cazares2023fast}, 
 is Devroye's algorithm adapted to the special case when the support of the density is the interval $(0,1)$.) 

\begin{breakablealgorithm}
\caption{Simulation from the law of $S_t|\{S_t<s\}$ where $S_t$ is defined in~\eqref{eq:def_tempered_stable} with $q=0$}
\label{alg:small_stable}
\begin{algorithmic}[1]
\Require{Parameters $\alpha\in(0,1)$, $\theta,t,s>0$}
\State{Set $\beta=\alpha/(1-\alpha)$ and generate $E'\sim\Exp(1)$}
\State{Generate $U$ from the density proportional to $\1_{(0,1)}(u)\exp\big(-\sigma_\alpha(u)(\theta t)^{\beta+1}s^{-\beta}\big)$ via~\cite[Algorithm 11]{cazares2023fast}}
\State{\Return $(s^{-\beta}+(\theta t)^{-\beta-1}E'/\sigma_\alpha(U))^{-1/\beta}$}
\end{algorithmic}
\end{breakablealgorithm}

Via an additional rejection step, we may sample tempered stable increments conditioned to be small.

\begin{breakablealgorithm}
\caption{Simulation from the law of $S|\{S_t<s\}$ where $S_t$ is defined in~\eqref{eq:def_tempered_stable} with $q>0$}
\label{alg:small_tempered_stable}
\begin{algorithmic}[1]
\Require{Parameters $\alpha\in(0,1)$, $\theta,q, t,s>0$}
\Repeat
    {~generate  $E\sim\Exp(1)$}
    \State{Generate $W$ from the law of $S_t |\{S_t < s\}$ via Algorithm~\ref{alg:small_stable} where $S$ is defined in~\eqref{eq:def_tempered_stable} with $q=0$}
\Until{$E>qW$}
\State{\Return $W$}
\end{algorithmic}
\end{breakablealgorithm}

The expected running time of both algorithms is as follows.

\begin{prop}
\label{prop:small_temperd_stable}
Let $S_t$ follow the law in~\eqref{eq:def_tempered_stable}.
Algorithms~\ref{alg:small_stable} and~\ref{alg:small_tempered_stable}
sample from the law of $S_t|\{S_t<s\}$ for $q=0$ and $q>0$, respectively, with average running time bounded above by a multiple $5\me^{qs}$.
\end{prop}

\begin{rem}
(I) If the truncation level 
$r$ 
satisfies $r<s$, then
the truncated tempered stable subordinator $Y$ with L\'evy measure $\nu_{r,q}$
in~\eqref{eq:Levy_measure_description},
conditioned to be smaller than $s$,
has the same law as 
$S_t|\{S_t<s\}$. This is because, if we define $Y$ by removing the jumps of $S$ greater than $r$, on the event $\{S_t<s\}$
we have 
$S_t=Y_t$~a.s. 
In particular, Algorithm~\ref{alg:small_tempered_stable} can be used to sample from he law of $Y_t|\{Y_t<s\}$.\\
\noindent (II) At first sight it may appear that the complexity of Algorithm~\ref{alg:small_tempered_stable} grows exponentially with $s$, which would be counter-intuitive since the conditioning distorts the law less if $s$ is large. In fact this is not so: in the final step of the proof of Proposition~\ref{prop:small_temperd_stable} in~\eqref{eq:final_bound_W} below, the variable $W$ is not only bounded above by a constant $s$ but also by $S_t$. We did not explicitly include this ($s$ independent) bound in the statement of the proposition as we are only interested in the small values of $s$.
\end{rem}

\begin{proof}[Proof of proposition~\ref{prop:small_temperd_stable}]
\underline{Case $q=0$.}
Denote by $g_t:\RP\to\RP$ 
the density of $S_t$ (for $t>0$). Note that $\beta/\alpha=\beta+1$ and hence $\sigma_\alpha(0+)\coloneqq\lim_{x\downarrow0}\sigma_\alpha(x)=(1-\alpha)\alpha^{\alpha/(1-\alpha)}$. The density $g_t$ of $S_t$ can be expressed as follows~\cite[Section~4.4]{uchaikin2011chance}:  for any $x>0$ we have 
\begin{equation}
\label{eq:zolotarev_representation}
    g_t(x)=\varphi_\alpha((\theta t)^{-1/\alpha}x)(\theta t)^{-1/\alpha},\quad\text{where}\quad
\varphi_\alpha(x)\coloneqq \beta
    \int_0^1 \sigma_\alpha(u)x^{-\beta-1}\me^{-\sigma_\alpha(u)x^{-\beta}}\md u.
\end{equation}
    The representation for the density shows that, if the uniform $U\sim\mathrm{U}(0,1)$ and exponential $E\sim\Exp(1)$ (with mean one) random variables are independent, then for every $x>0$ we have
\[
\p[S_t\le (\theta t)^{1/\alpha}x]
=\int_0^1 \me^{-\sigma_\alpha(u)x^{-\beta}}\md u
=\p[E\ges\sigma_\alpha(U)x^{-\beta}]
=\p[(\sigma_\alpha(U)/E)^{1/\beta}\les x].
\]
This implies
$S_t
\eqd (\theta t)^{1/\alpha}(\sigma_\alpha(U)/E)^{1/\beta}$
and hence 
\begin{align}
S_t\,|\,\{S_t< s\}
&\eqd (\theta t)^{1/\alpha}(\sigma_\alpha(U)/E)^{1/\beta}\,\big|\, \{\sigma_\alpha(U)(\theta t)^{\beta+1}s^{-\beta}<E\}.\label{eq:joint_small_stable}
\end{align}

The vector $(U,E)$, conditional on the event  $\{\sigma_\alpha(U)(\theta t)^{\beta+1}s^{-\beta}<E\}$, 
has a joint density $f_{U,E}(u,x)$,
proportional to the function  $(u,x)\mapsto\me^{-x}\1_{(0,1)}(u)\1_{(\sigma_\alpha(u)(\theta t)^{\beta+1}s^{-\beta},\infty)}(x)$.
The density $f_U(u)$ 
of the marginal law $U\big|\{\sigma_\alpha(U)(\theta t)^{\beta+1}s^{-\beta}<E\}$
is thus proportional 
to
\[
u\mapsto 
 \1_{(0,1)}(u)
\int_{\sigma_\alpha(u)(\theta t)^{\beta+1}s^{-\beta}}^\infty\me^{-x}\md x 
=\1_{(0,1)}(u)\me^{-\sigma_\alpha(u)(\theta t)^{\beta+1}s^{-\beta}}.
\]
The density $f_{E|U}(x|u)$ of $E\big|U, \{\sigma_\alpha(U)(\theta t)^{\beta+1}s^{-\beta}<E\}$
is clearly that of a shifted exponential and 
satisfies 
$f_{U,E}(u,x)=f_U(u)f_{E|U}(x|u)$.
Hence, to sample $S_t\,|\,\{S_t<s\}$ we need only sample $U$ from the marginal density $f_U$, sample an independent $E'\sim\Exp(1)$, set $E=E'+\sigma_\alpha(U)(\theta t)^{\beta+1}s^{-\beta}$ and compute 
\[
S_t=(\theta t)^{1/\alpha}(\sigma_\alpha(U)/E)^{1/\beta}
=(s^{-\beta}+(\theta t)^{-\beta-1}E'/\sigma_\alpha(U))^{-1/\beta}.
\]

In the previously described procedure, the only nontrivial step is the simulation of $U$ with density $f_U$. Since $\sigma_\alpha$ is a convex function~\cite[Lemma 4.3]{cazares2023fast}, we may apply Devroye's rejection sampling algorithm~\cite{devroye2012note} (which does not require the normalisation constant of $f_U$) with acceptance probability no less than $1/5$, see~\cite[Section~4.4]{devroye2012note}.

\noindent \underline{Case $q>0$.}
According to~\cite[Lem.~B.1]{cazares2023fast}, Algorithm~\ref{alg:small_tempered_stable} samples from the law of $S_t|\{S_t<s\}$. Since $W<s$~a.s.,  the acceptance probability is bounded below as follows
\begin{equation}
\label{eq:final_bound_W}
\p[E>qW]=\E[\p[E>qW| W]]
=\E[\me^{-qW}]\ges\me^{-qs}.
\end{equation}
The expected running time of Algorithm~\ref{alg:small_tempered_stable}, proportional to $5/\p[E>qW]$, is bounded above as claimed.
\end{proof}

We conclude this section by noting that a naive sampling method, which simulates 
a stable variable 
$S_t$ until $S_t<s$, can be  highly inefficient in our context.
By~\eqref{eq:zolotarev_representation}, acceptance probability equals
\[
\p[S_t<s]=\p[S_1<t^{-1/\alpha}s]=\int_0^{t^{-1/\alpha}s}\varphi_\alpha((\theta t)^{-1/\alpha}x)(\theta t)^{-1/\alpha}\md x.
\]
Since all derivatives of 
$\varphi_\alpha$ at zero equal $0$, the function $\varphi_\alpha$ increases extremely slowly near zero (for instance, $\varphi_{0.3}(10^{-4})\approx 10^{-20}$), making the acceptance probability tiny for small $s$.
The naive algorithm has finite expected running time for any fixed level $s$ and time $t$. However, in line~\ref{step:small_tempered_stable} of Algorithm~\ref{alg:triplet_Z}, both time and level are random,
making the expected running time possibly infinite, cf. Appendix~\ref{app:Chi}.

\section{Applications}\label{sec:applications}

The implementation of Algorithm~\ref{alg:triplet_Z} in Python and Julia, used in the numerical examples in this section, is in the GitHub repository~\cite{repository}.

\subsection{How to choose the truncation parameter $r$?}
\label{subsec:Implementation}

As discussed in Section~\ref{section:result}, just above Theorem~\ref{thm:main_complexity},
the truncation parameter $r$ in~\eqref{eq:Levy_measure_description}
is in an admissible interval $(0,r_0]$, 
where $r_0$ could be infinite. 
The total mass $\Lambda_r < \infty$ of the compound Poisson component of the subordinator $Z$ satisfies the following: $\Lambda_r\searrow \Lambda_{r_0}\in[0,\infty)$ as $r\rightarrow r_0$
and $\Lambda_r\nearrow \infty$ as $r\rightarrow 0$.
Moreover, by~\eqref{eq:lambda_r} we know 
$\Lambda_r
\les
\Lambda_{r_0}+\vartheta(r^{-\alpha}-r_0^{-\alpha})/\alpha$.
Since the bound on the expected running time in Theorem~\ref{thm:main_complexity}
is proportional to $\Lambda_r$
and the inequality 
$(2\alpha)^{-\alpha}\les 2$ holds 
for all $\alpha\in(0,1)$,
 picking the truncation parameter $r$ proportional to $2\alpha$ in the regime $\alpha\to 0$ makes   $\Lambda_r$ explode at most as fast as $1/\alpha$.
The bound on the expected running time in~\eqref{eq:main_complexity} of Theorem~\ref{thm:main_complexity} suggest that the truncation parameter $r$ should be inversely proportional to 
 the tempering parameter $q$. If this were not the case, for example the term $\me^{q\min\{c_0,r\rho\}}$ in the bound in~\eqref{eq:main_complexity} on the expected complexity would deteriorate (possibly exponentially) as $q$ becomes large.

 A simple choice, satisfying the aforementioned requirements, is to set $r=\min\{r_0,2\alpha/q\}$. 
 This choice of the truncation level, while not optimal, is theoretically sound as discussed above, easy to compute and  appears to perform well in practice, see Figure~\ref{fig:implematation_running_time}.
We note that, in addition to being a truncation parameter in~\eqref{eq:Levy_measure_description}, $r$ in Algorithm~\ref{alg:triplet_Z} is also used to cap the boundary function $b$ in each iteration of Algorithm~\ref{alg:triplet_Z}. This is analogous to the role played by the parameter $R$ in~\cite[Alg.~2]{cazares2023fast}:
 based on the bound in~\cite[Cor.~2.2]{cazares2023fast}, we took $R$ to be  proportional to $\alpha/q$ in~\cite[line~1, Alg.~2]{cazares2023fast}, which is consistent with our choice of $r$ above.

\begin{figure}
\centering
\begin{subfigure}{0.32\textwidth}
    \includegraphics[width=\textwidth]{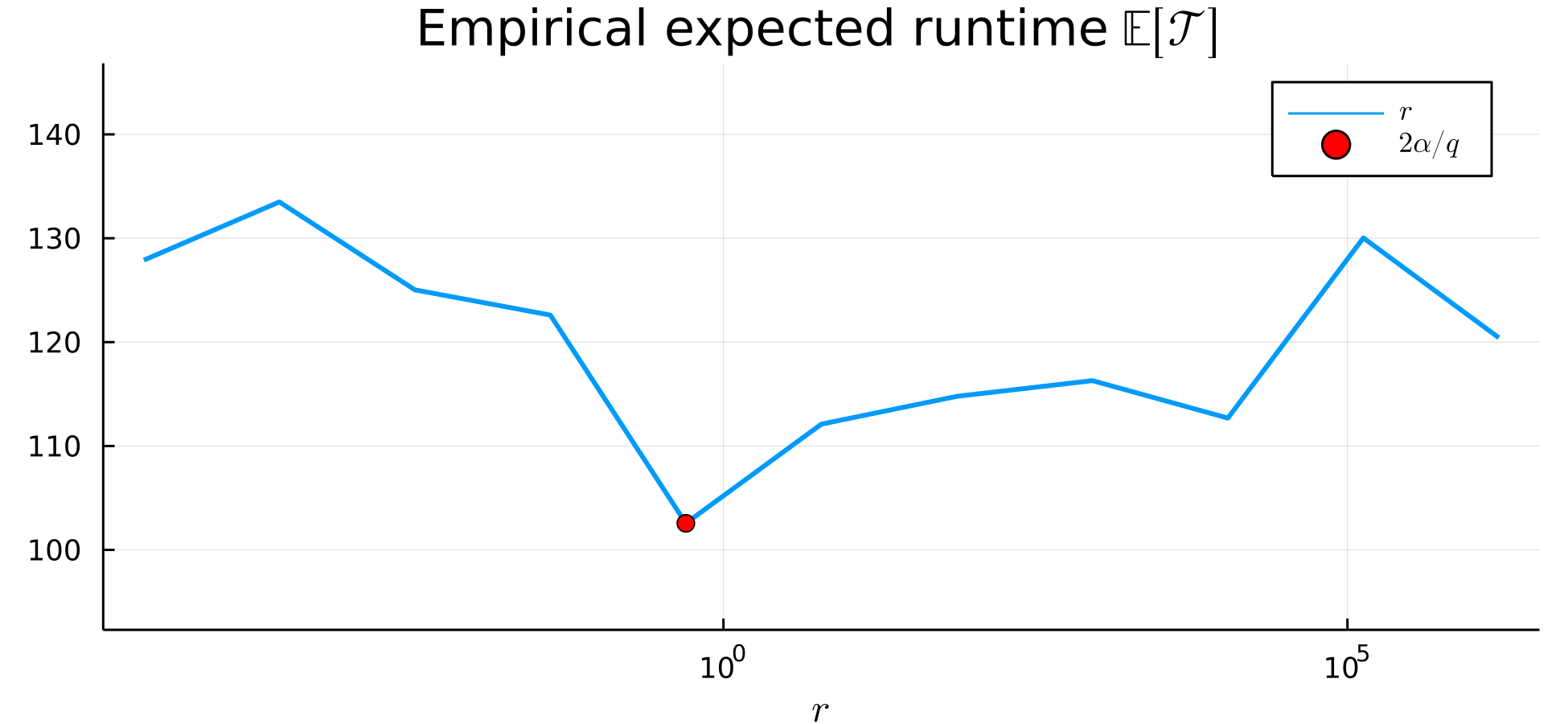}
    \caption{$\alpha=0.25,\, q=1$.}
    \label{fig:r_0=10q=1}
\end{subfigure}
\hfill
\begin{subfigure}{0.32\textwidth}
    \includegraphics[width=\textwidth]{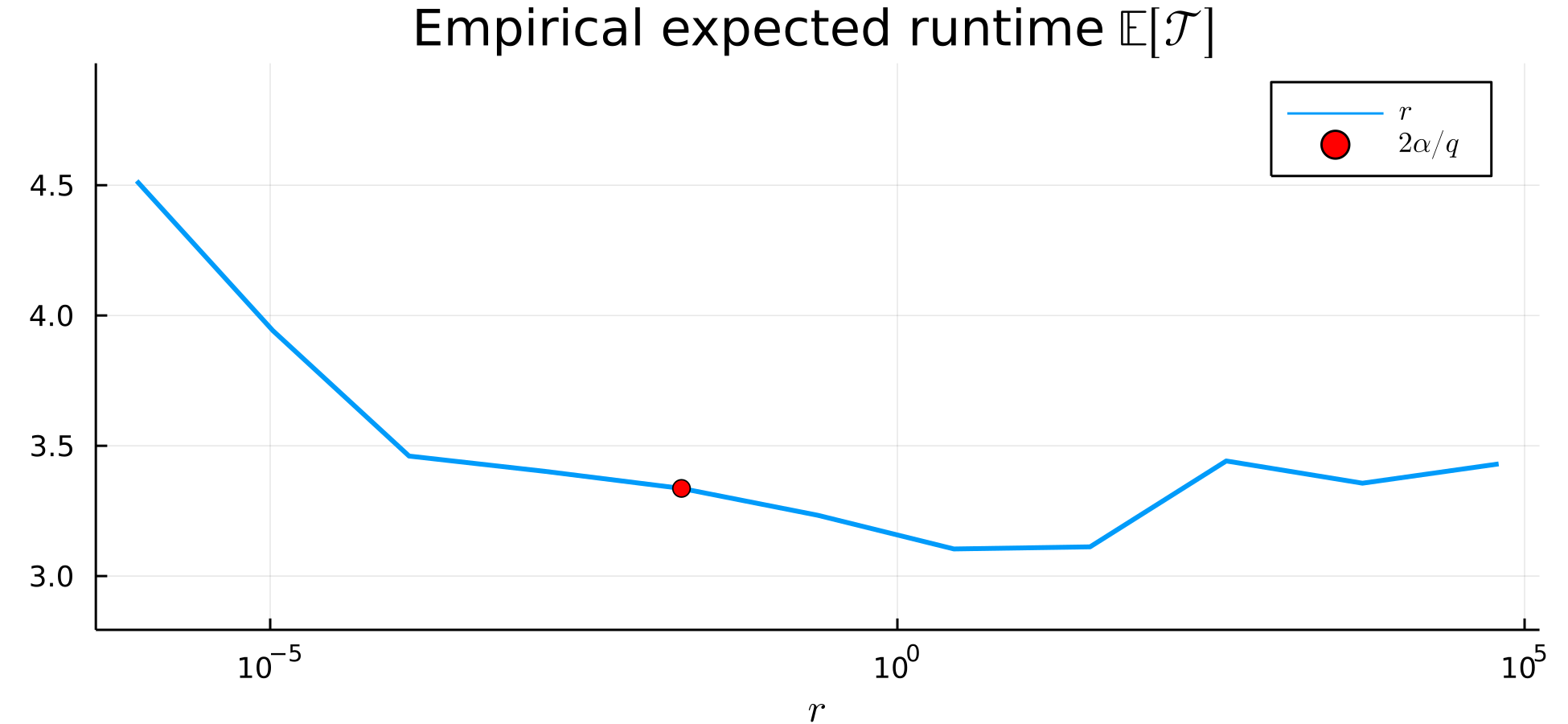}
    \caption{$\alpha=0.95,\, q=100$.}
    \label{fig:r_0=inftyq=10}
\end{subfigure}
\hfill
\begin{subfigure}{0.32\textwidth}
    \includegraphics[width=\textwidth]{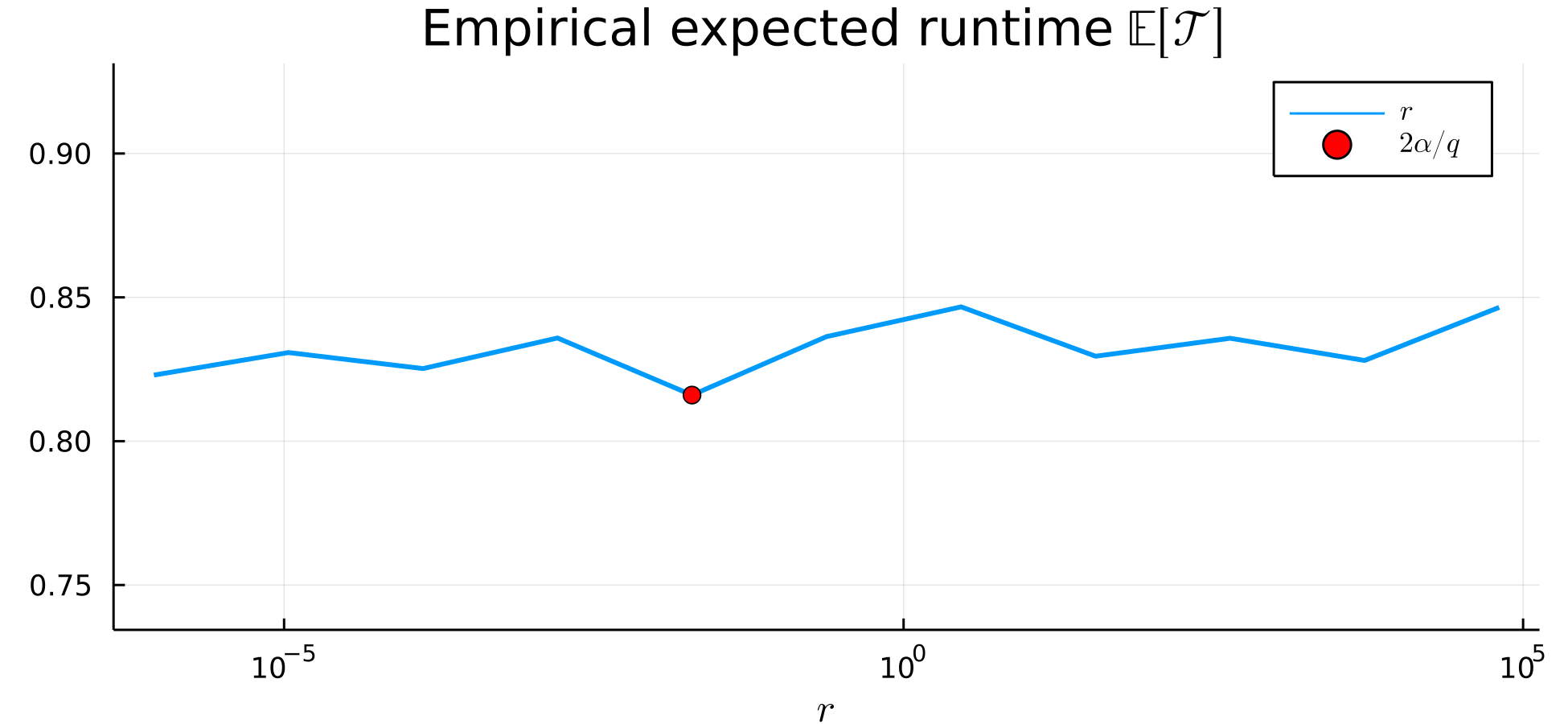}
    \caption{$\alpha=0.98,\, q=100$.}
    \label{fig:r_0=inftyq=100}
\end{subfigure}
        
\caption{The pictures show an implementation of Algorithm~\ref{alg:triplet_Z} with parameters $\vartheta=2$, $r_0=\infty$, $c(t)\equiv 5$, $\lambda_{r_0}(\md x)=\me^{-x}\md x$. The average running time $\E[\mathcal T]$ is measured in seconds taken for every $10^4$ samples.}
\label{fig:implematation_running_time}
\end{figure}


\subsection{Monte Carlo approximation for a fractional partial differential equations}
\label{subsec:Example_FPDE}

Let $g:[0,\infty)\times \R^d\to  \R$ be a continuous function that vanishes at infinity and $\phi:\R^d\to\R$.
For $a<b$, consider the following fractional partial differential equation (FPDE):
\begin{equation}
\label{eq:fpde}
\begin{split}
    (-_tD_a+A_x)u(t,x)&=-g(t,x) \quad (t,x)\in(a,b]\times\mathbb R^d,\\
    u(a,x)&=\phi(x), \quad x\in \mathbb R^d,
\end{split}
\end{equation}
where $A_x$ is a generator of a Feller semigroup on $C_\infty(\mathbb R^d)$ acting on $x$, $\phi\in\text{Dom}(A_x)$, the operator $-_tD_a$ is a generalised differential operator of Caputo type of order less than $1$ acting on the time variable $t\in[a,b]$.
The solution $u$ of the FPDE problem in~\eqref{eq:fpde} exists and its stochastic representation~\cite{hernandez2017generalised},
\begin{equation}
\label{eq:fpde_solution}
u(t,x)=\E\left[\phi(X_{T_t}^x)+\int_0^{T_t}g(\zeta_s^{a,t},X_s^x)\md s \right],
\end{equation}
where $\{X_s^x\}_{s\ges0}$ is the stochastic process generated by $A_x$,  $\{\zeta_s^{a,t}\}_{s\ges0}$ is the decreasing $[a,b]$-valued  process started at $t\in[a,b]$ and generated by $-_tD_a$
and $T_t=\inf\{s>0,\zeta_s^{a,t}<a\}$. 
Note that $T_t$
is the first-passage time of the subordinator 
$(t-\zeta_s^{a,t})_{s\in[0,\infty)}$
over the constant level $t-a$.

 Consider the Ornstein--Uhlenbeck process $X$ on $\R^2$ and let $\phi(x)=x_1+x_2^2$ for $x=(x_1,x_2)\in\R^2$. More precisely, let $X^x$ be the solution of the stochastic differential equation
\begin{equation}
\label{eq:sde_OU}
\md X_s=-\mu X_s\md s+\gamma\md W_s,
\quad X^x_0=x,
\qquad \text{where }
\mu=\big(\begin{smallmatrix}1 0\\ 0 1\end{smallmatrix}\big),\enskip
\gamma =\big(\begin{smallmatrix}1 1\\ 0 1\end{smallmatrix}\big),
\end{equation}
$W$ is a standard Brownian motion on $\R^2$. Let the generalised Caputo operator  $-_tD_a$ be given by
\[
-_tD_af(t):=-\int_0^{t-a}(f(t-s)-f(t))\nu(\md s)+(f(a)-f(t))\int_{t-a}^\infty\nu(\md s).
\]
where 
\[
\nu(\md s)=\frac{1}{-\Gamma(-0.65)}\1\{0<s<1\}\me^{-s}s^{-1.65}\md s+\1\{s\ges 1\}s^{-5}\md s.
\]
The driftelss subordinator $t-\zeta^{a,t}$ has Levy measure of the form~\eqref{eq:Levy_measure_description}
with $\alpha= 0.65$, $q= 1$,  the truncation parameter $r= 1$ and  the finite activity component is given by $\lambda_1(\md s)=\1\{s\ges 1\}s^{-5}\md s$.

For any time $t>0$, the  marginal $X_t$ of the Ornstein--Uhlenbeck process  is Gaussian with known mean and variance. Using Algorithm~\ref{alg:triplet_Z}, together with the representation in~\eqref{eq:fpde_solution}, we simulate the time $T_t$ and plot in
Figure~\ref{fig:fpde_ou} the Monte Carlo estimator (see e.g.~\cite{MR4219829} for the definition of the estimator) to the solution of the FPDE in~\eqref{eq:fpde} at time $t-a=5$.

\begin{figure}
\centering
\includegraphics[width=.49\textwidth]{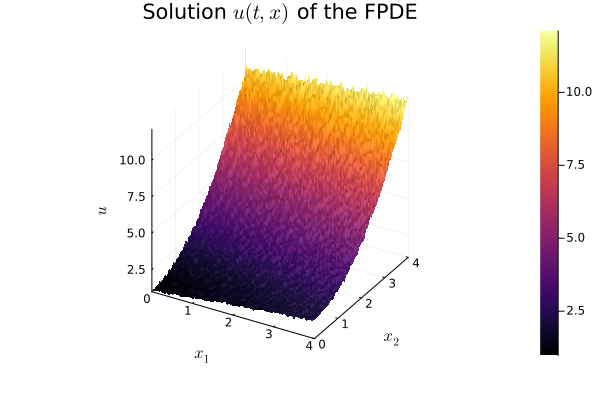}
\caption{A Monte Carlo estimator 
 $u_n(t,x)$ of the solution $u(t,x)$ of the FPDE in~\eqref{eq:fpde_solution} at time $t-a=5$,
 based on the representation in~\eqref{eq:fpde_solution}. This computation is based on  $n=10^4$ samples. 
}
\label{fig:fpde_ou}
\end{figure}

\section{Proof of Theorem~\ref{thm:main_complexity}}
\label{sec:proof}
 Recall that $Z=Y+Q$, where the truncated tempered stable subordinator $Y$ and the compound Poisson process $Q$ are independent with L\'evy  measures $\nu_{r,q}$ and $\lambda_r$ in~\eqref{eq:Levy_measure_description}, respectively.
 When convenient, we will assume that  $Y$ is obtained from 
a tempered stable subordinator $S$ (with L\'evy measure $\nu_q$ given in~\eqref{eq:def_tempered_stable}), by removing all jumps of $S$ of size greater than $r$.

\subsection{Algorithm~\ref{alg:triplet_Z} samples 
the first-passage event of the subordinator $Z$}\label{sec:proof_alg}
We now prove that Algorithm~\ref{alg:triplet_Z}
simulates from the law of the random vector
$\chi_c^Z=\big(\tau_c^Z,
    Z_{\tau_c^Z-},
    Z_{\tau_c^Z}\big)$.
For an auxiliary parameter $\rho\in(0,1)$, the algorithm defines a new boundary function
$b(t):=\min\{r\rho, c(t)\}$, bounded above by the truncation parameter $r$. 
 Algorithm~\ref{alg:triplet_Z} then samples two quantities: 
the first jump time $D$ of the compound Poisson process $Q$ (with L\'evy measure $\lambda_r$, see representation of the L\'evy measure of $Z$ in~\eqref{eq:Levy_measure_description})
following exponential law with mean $1/\Lambda_r$ \textit{and} the vector $\chi_b^{Y}\coloneqq(\tau_b^{Y},Y_{\tau_b^{Y}-},Y_{\tau_b^{Y}})$, where $Y$ is the truncated tempered stable process with L\'evy measure $\nu_{r,q}$ in~\eqref{eq:Levy_measure_description}. 

Algorithm~\ref{alg:triplet_Z} then compares times $\tau_b^Y$ and $D$. Since $D$ and $\tau_b^Y$ are independent and $D$ has a density, the two times are not equal almost surely.  Set $t=\tau_b^{Y}\wedge D$ and consider two cases: 

\noindent \textbf{(i)} If $t=\tau_b^{Y}<D$, then we obtained a sample from $(t,Z_{t-},Z_t)=\chi_b^{Y}$, where $Z_{t-}\les b(t)$ and $Z_t\ges b(t)$. 
Algorithm~\ref{alg:triplet_Z} then redefines
$b(\cdot)\leftarrow\min\{r\rho, c(\cdot+t)-Z_t\}$.
\textbf{(i-a)} If the sampled vector $(t,Z_{t-},Z_t)$ satisfies $b(0)\les 0$ (i.e. $Z_t\ges c(t)$), the algorithm has already produced the desired output $\chi_c^Z$. \textbf{(i-b)} If $b(0)>0$ (i.e. $Z_t< c(t)$), Algorithm~\ref{alg:triplet_Z} sets $D\leftarrow D-t$ and repeats the procedure starting from 
the sampled vector 
$(\tau_b^{Y},Y_{\tau_b^{Y}-},Y_{\tau_b^{Y}})$
 with the redefined boundary function $b$. 
 It is important to note that in the case~(i-b), the undershoot $Y_{\tau_b^{Y}-}$ has been essentially discarded as it does not feature in the final output of the algorithm, cf. lines~\ref{step:main_4} and~\ref{step:main_14} in Algorithm~\ref{alg:triplet_Z} and Figure~\ref{fig:consequtive_crossing}.
 
 \noindent \textbf{(ii)} If $t=D<\tau_b^{Y}$, Algorithm~\ref{alg:triplet_Z} simulates the first jump $J$ of $Q$, which follows the law $\lambda_r(\md x)/\Lambda_r$ \textit{and} $Y_t$ conditional on $\{\tau_b^Y>t\}=\{Y_t\les b(t)\}$. This  yields a sample from $(t,Z_{t-},Z_t)=(D,Y_t,Y_t+J)$ satisfying $Z_{t-}\les b(t)$. Note here that, with probability one, $t$ is not a jump time of $Y$, since $D$ has a density and is independent of $Y$. Thus $Z_t-Z_{t-}=Q_t-Q_{t-}=J$. Moreover,  since we are on the event $\{\tau_b^Y>t\}$ and $t=D$ is the time of the first jump of $Q$, we must have $Q_{t-}=0$ and hence $Z_{t-}=Y_t$. Algorithm~\ref{alg:triplet_Z} then redefines
$b(\cdot)\leftarrow\min\{r\rho, c(\cdot+t)-Z_t\}$. There are again two possibilities.
 \textbf{(ii-a)} If $b(0)\les0$ (i.e. $Z_t\ges c(t)$), the algorithm has already produced the desired vector $\chi_c^Z$. \textbf{(ii-b)} If $b(0)>0$ (i.e. $Z_t< c(t)$), then Algorithm~\ref{alg:triplet_Z} resamples the first jump time $D$ and repeats the procedure starting from 
 $(t,Z_{t-},Z_t)$. As in case~(i-b), the undershoot $Z_{t-}$ is ignored as it does not play a role in the final output, 
 see lines~\ref{step:main_4} and~\ref{step:main_14}
 in Algorithm~\ref{alg:triplet_Z}.
 
It follows from the definition and the update structure of the boundary function $b(t)$ that, in each iteration, Algorithm~\ref{alg:triplet_Z} either produced a sample of $\chi_c^Z$ or reduces the initial value $c_0=c(0)$ by at least $r\rho$. The procedure described above therefore terminates in a finite number of steps (see Section~\ref{sec:proof_time} which describes and proves an upper bound on the expected running time).

Finally we note that  Algorithm~\ref{alg:triplet_Y}
indeed samples from the law of the triplet $(\tau_b^{Y},Y_{\tau_b^{Y}-},Y_{\tau_b^{Y}})$ of the truncated tempered stable process $Y$. Since the boundary function $b(t)$ is smaller than the truncation parameter $r$,  all jumps of the tempered stable process $S$, prior to the crossing time of $b$, are smaller than $r$, making $S_s=Y_s$
for all $s\in[0,\tau_b^S)$. 
The only jump that might be greater than $r$
is the one that occurs at the crossing time $\tau_b^S$, which 
Algorithm~\ref{alg:triplet_Y}
corrects for in line~\ref{step:accept_jump<r},
thus producing a sample from the law of $(\tau_b^{Y},Y_{\tau_b^{Y}-},Y_{\tau_b^{Y}})$.
Similarly, 
as noted in Remark~(I) after Proposition~\ref{prop:small_temperd_stable},
since the boundary function $b$ is bounded by the truncation parameter $r$, line~\ref{step:small_tempered_stable} in Algorithm~\ref{alg:triplet_Z} samples from the conditional law $Y_t|\{Y_t\les b(t)\}$.

Having established that Algorithm~\ref{alg:triplet_Z} samples from the law of $\chi_c^Z=\big(\tau_c^Z,
    Z_{\tau_c^Z-},
    Z_{\tau_c^Z}\big)$,
    in the following section we study its computational complexity.

\subsection{Computational complexity of Algorithm~\ref{alg:triplet_Z}}\label{sec:proof_time}

Recall the parameters in Algorithm~\ref{alg:triplet_Z}:
 $\alpha,\rho\in(0,1)$, $\vartheta\in(0,\infty)$, $r\in(0,\infty]$, $q\in[0,\infty)$, finite measure $\lambda_r$ in~\eqref{eq:Levy_measure_description} and a non-increasing function $c:[0,\infty)\to[0,\infty)$ with $c_0=c(0)>0$.
 We start with preliminary lemmas needed in the proof of Theorem~\ref{thm:main_complexity}.
 Lemma~\ref{lem:local_algs_complexity} bounds the complexity of Algorithms~\ref{alg:triplet_Y} and~\ref{alg:small_tempered_stable}, called within Algorithm~\ref{alg:triplet_Z}.

\begin{lem}
\label{lem:local_algs_complexity}
The following statements hold.
\begin{itemize}
\item[(a)] The expected running time of Algorithm~\ref{alg:triplet_Y}, called in line~\ref{step:fpe_truncated_tempered_stable} of Algorithm~\ref{alg:triplet_Z}, is bounded by a multiple of $\mathcal{T}_0(\alpha,\vartheta,q,\min\{c_0,r\rho\})/(1-\rho^\alpha)$, where $\mathcal{T}_0$ is the bound on the complexity of~\nameref{alg:FPTS}. 
\item[(b)] The running time of Algorithm~\ref{alg:small_tempered_stable}, called in line~\ref{step:small_tempered_stable} of Algorithm~\ref{alg:triplet_Z}, is a constant multiple of a geometric random variable with acceptance probability bounded below by $\me^{-q\min\{c_0,r\rho\}}/5$.
\end{itemize}
The multiplicative constants, suppressed in (a) and (b), 
do not depend on the parameters in Algorithm~\ref{alg:triplet_Z}.
\end{lem}

\begin{proof}
(a) The acceptance probability in 
line~\ref{step:accept_jump<r} of Algorithm~\ref{alg:triplet_Y}
equals 
$\p[ \Delta_b^S \les r]$, where $\Delta_b^S:=S_{\tau_b^S}-
    S_{\tau_b^S-}$ and 
$S$ is a tempered stable subordinator with L\'evy measure $\nu_q(\md x)$ in~\eqref{eq:def_tempered_stable}, proportional to $\me^{-qx}x^{-\alpha-1}\md x$.
It is well-known (see e.g. Eq.~(A.1a) in~\cite[Thm.~A.1]{cazares2023fast})
that 
$$
\p\big[\Delta_b^S> r\big]=\int_{(0,\infty)^2}
    \1_{\{u\les b(t)\}}
    \nu_q([(b(t)-u)\vee r,\infty))\md t\,\p[S_t\in\md u].
$$
Moreover,
$\nu_q([(b(t)-u)\vee r,\infty))\les \nu_q([ r,\infty))$ and 
$\nu_q([b(t)-u,\infty))\ges \nu_q([r\rho,\infty))$, since $b(t)-u\les b(t)\les r\rho$ (by construction in line~\ref{step:fpe_truncated_tempered_stable} of Algorithm~\ref{alg:triplet_Z})
and  $x\mapsto \nu_q([x,\infty))$ is decreasing. Thus we obtain
\begin{align*}
\p\big[\Delta_b^S> r\big]
\les\p\big[\Delta_b^S> r\,\big|\,\Delta_b^S>0\big]
&=\frac{\int_{(0,\infty)^2}
    \1_{\{u\les b(t)\}}
    \nu_q([r,\infty))\md t\,\p[S_t\in\md u]}{\int_{(0,\infty)^2}
    \1_{\{u\les b(t)\}}
    \nu_q([b(t)-u,\infty))\md t\,\p[S_t\in\md u]}\\
&\les \frac{\nu_q([r,\infty))}{\nu_q([r\rho,\infty))}
=\rho^\alpha\frac{\int_r^\infty \me^{-qx}x^{-\alpha-1}\md x}{\int_{r}^\infty \me^{-q\rho x}x^{-\alpha-1}\md x}\les \rho^\alpha,
\end{align*}
implying 
$\p[ \Delta_b^S \les r]=1-\p[ \Delta_b^S > r]\ges 1-\rho^\alpha$. Thus, the expected running time of  Algorithm~\ref{alg:triplet_Y}, called in line~\ref{step:fpe_truncated_tempered_stable} of Algorithm~\ref{alg:triplet_Z},
is bounded above by a multiple of $1/(1-\rho^\alpha)$ as claimed. 

\noindent (b) This is a direct consequence of 
Proposition~\ref{prop:small_temperd_stable}
and the fact $b(D)=\min\{c(D),r\rho \}\les \min\{c_0,r\rho \}$. 
\end{proof}


The following lemma will be used to bound the complexity of Algorithm~\ref{alg:triplet_Z} in terms of the complexities of Algorithms~\ref{alg:triplet_Y} and~\ref{alg:small_tempered_stable}, called within Algorithm~\ref{alg:triplet_Z}. The key feature of Lemma~\ref{lem:comp_cost} is that it allows the random complexities $R_n$
to be dependent (as is the case when running Algorithm~\ref{alg:triplet_Z}).

\begin{lem}
\label{lem:comp_cost}
Let $(R_n)_{n\in\N}$ be a sequence of non-negative random variables adapted to the filtration $(\mathcal{G}_n)_{n\in\N}$ and let $M$ be an a.s. finite $(\mathcal{G}_n)$-stopping time.\footnote{By definition, $\{M=n\}\in\mathcal{G}_n$ for all $n\in\N$.
} If there exists some $C>0$ such that $\E[R_n|\mathcal{G}_{n-1}]\les C$ a.s. for all $n\in\N$ and $\E[M]<\infty$, then $\E[\sum_{n=1}^M R_n]\les C\E[M]$. Moreover, if for some $p\ges 1$ there exist $C'>0$ and $\epsilon>0$, such that  $\E[R_n^{2p}]\les C'$ for all $n\in\N$ and $\E[M^{2p+\epsilon}]<\infty$, then $\E[(\sum_{n=1}^M R_n)^p]<\infty$.
\end{lem}

\begin{proof}
The first claim follows easily from Fubini's theorem and the fact $\E[M]=\sum_{n\in\N} \p[M\ges n]$:
\[
\E\Bigg[\sum_{n=1}^M R_n\Bigg]
=\E\Bigg[\sum_{n=1}^\infty \1_{\{M\ges n\}}R_n\Bigg]
=\sum_{n=1}^\infty \E\left[\1_{\{M\ges n\}}\E[R_n|\mathcal{G}_{n-1}]\right]
\les C\sum_{n=1}^\infty \p[M\ges n]=C\E[M].
\]
Recall that for any $x_1,\ldots,x_n\in[0,\infty)$,
 we have $(\sum_{i=1}^n x_i)^p\les \sum_{i=1}^n n^{p-1} x_i^p$. (Since $p\ges 1$, Jensen's inequality  yields $(\sum_{i=1}^n x_i/n)^p\les \sum_{i=1}^n x_i^p/n$.)  
 Since $\1_{\{M\ges n\}}M^{p-1}\les M^{p+\epsilon/2}/n^{1+\epsilon/2}$,  the Cauchy--Schwarz inequality implies
\begin{align*}
\E\Bigg[\Bigg(\sum_{n=1}^M R_n\Bigg)^p\Bigg]
&= \E\left[\sum_{n=1}^M M^{p-1} R_n^p\right]  \les\sum_{n=1}^\infty \E[\1_{\{M\ges n\}}M^{p-1}R_n^p]
\les\sum_{n=1}^\infty n^{-1-\epsilon/2}\E[M^{p+\epsilon/2}R_n^p]\\
&\les\sum_{n=1}^\infty n^{-1-\epsilon/2}\left(\E[M^{2p+\epsilon}]\E[R_n^{2p}]\right)^{1/2}
\les \sqrt{C'\E[M^{2p+\epsilon}]}\sum_{n=1}^\infty n^{-1-\epsilon/2}<\infty.
\qedhere
\end{align*}
\end{proof}

The following elementary lower bound will be used to bound the expected number of jumps 
(before the crossing time of $Z$)
of the compound Poisson process $Q$.  

\begin{lem}
\label{lem:bound_psi_u}
Pick any $q\in[0,\infty)$, $r\in(0,\infty]$, $\alpha\in(0,1)$, $u\in[0,\infty)$ and recall the definition of $\Upsilon(u,r,q)$
in~\eqref{eq:upsilon_function}.
It holds that
\begin{equation*}
\int_0^r(1-\me^{-ux})\me^{-qx}x^{-\alpha-1}\md x
\ges\Upsilon(u,r,q).
\end{equation*}
\end{lem}

\begin{proof}
Denote $s=\min\{r,1\}$. Since $1-\me^{-t}\ges \tfrac{1}{2}\min\{t,1\}$ for $t>0$, we have
\begin{multline*}
\int_0^r(1-\me^{-ux})\me^{-qx}x^{-\alpha-1}\md x
\ges \int_0^s (1-\me^{-ux})\me^{-qx}x^{-1}\md x
= \int_0^s \int_q^{q+u}\me^{-xy}\md y\md x\\
= \int_q^{q+u} (1-\me^{-sy})y^{-1}\md y
\ges \frac{1}{2}\int_q^{q+u} \min\{y^{-1},s\}\md y
=\frac{1}{2}\begin{cases}
    us,&u+q<1/s,\\
    1-sq+\log((u+q)s),&q<1/s\les u+q,\\
    \log(1+u/q),&1/s\les q.
\end{cases}
\end{multline*}
Since $1/s=\max\{r^{-1},1\}$, the lemma follows.
\end{proof}

\begin{proof}[{\textbf{{Proof of Theorem~\ref{thm:main_complexity}}}}]

We now analyse the computational complexity of Algorithm~\ref{alg:triplet_Z}. We start by identifying all the variables sampled by the algorithm in terms of the subordinators  $Y$ and $Z$ and the compound Poisson process $Q$. 

Let  $K\ge 0$ be the number of jumps of the compound Poisson process $Q$
during the time interval $[0,\tau_c^Z]$.
Denote by
$T_1,\ldots,T_K$ the corresponding jump times of $Q$ in this interval and 
set $T_0\coloneqq 0$ and $T_{K+1}\coloneqq\tau_c^Z\ges T_K$. Consider now the consecutive crossing times of $Y$ (over the boundary $b$, defined in Algorithm~\ref{alg:triplet_Z}) between jump times of $Q$. For each $i\in\{0,\ldots,K\}$, set $T_{i,0}\coloneqq T_i$ and recursively define 
\[
T_{i,j+1}:=\inf\big\{t>T_{i,j}\,
    :\,Y_t-Y_{T_{i,j}}
        >\min\{c(t)-Z_{T_{i,j}},r\rho\}\big\},
\quad j=0,1,\ldots
\]
Observe that, for any $i\in\{0,\ldots,K\}$, the sequence $(T_{i,j})_{j\ges 0}$ becomes constant almost surely: this is because at each $j\ges0$,
the inequalities 
$Z_{T_{i,j+1}}-Z_{T_{i,j}}\ges Y_{T_{i,j+1}}-Y_{T_{i,j}} \ges \min\{c(T_{i,j+1})-Z_{T_{i,j}},r\rho\}$ ensure that for some $j_0\in\N$  we have $Z_{T_{i,j_0+1}}\ges c(T_{i,j_0+1})$, at which time the minimum in the definition in the previous display becomes zero, making the sequence of times constant after $j_0$.

Denote by $L_i\coloneqq\sum_{j=1}^\infty \1_{\{T_{i,j}<T_{i+1}\}}\ges 0$ 
the number of times $Y$ crosses the function $b$ 
during the interval  $[T_i,T_{i+1}]$. 
Crucially, for any 
$i\in\{0,\ldots,K\}$,
at each crossing time $T_{i,j}$, satisfying  $T_{i,j}<T_{i+1}$,  line~\ref{step:main_5} of Algorithm~\ref{alg:triplet_Z} reduces the barrier function $b$ by at least $r\rho$.
 Thus, we obtain 
$\sum_{i=0}^K L_i\les\lceil c_0/(r\rho)\rceil$ (here we use the notation $\lceil x\rceil\coloneqq\inf\{n\in\Z:n\ges x\}$). 
Put differently, Algorithm~\ref{alg:triplet_Z} samples the triplet $\chi_b^Y$ at times $T_{i,j}$ in the following sequence of times 
\begin{equation}
    \label{eq:number_intervals}
0=T_0<T_{0,1}...<T_1 < T_{1,1} < \ldots < T_{1,L_1}
< T_2 < T_{2,1} < \ldots 
< T_K < T_{K,1} < \ldots 
< T_{K,L_K} = \tau_c^Z,
\end{equation}
before returning the sample from the law of
 $\chi^Z_c=\chi^Z(\tau_c^Z)$ (see Figure~\ref{fig:consequtive_crossing}).

Based on Lemma~\ref{lem:comp_cost}, we will reduce the analysis of the running time of Algorithm~\ref{alg:triplet_Z} to controlling the number of ``loops'' (i.e. intervals between consecutive times in~\eqref{eq:number_intervals}) $M\coloneqq K+\sum_{i=1}^KL_i$, bounded above by $K+\lceil c_0/(r\rho)\rceil\ges M$, and bounding the number of operations required in each loop, i.e. between times $T_{i,j}$ and $T_{i,j+1}$ or between times $T_{i,L_i}$ and $T_{i+1}$.

\begin{figure}[ht]
\centering
\includegraphics[scale=0.3]{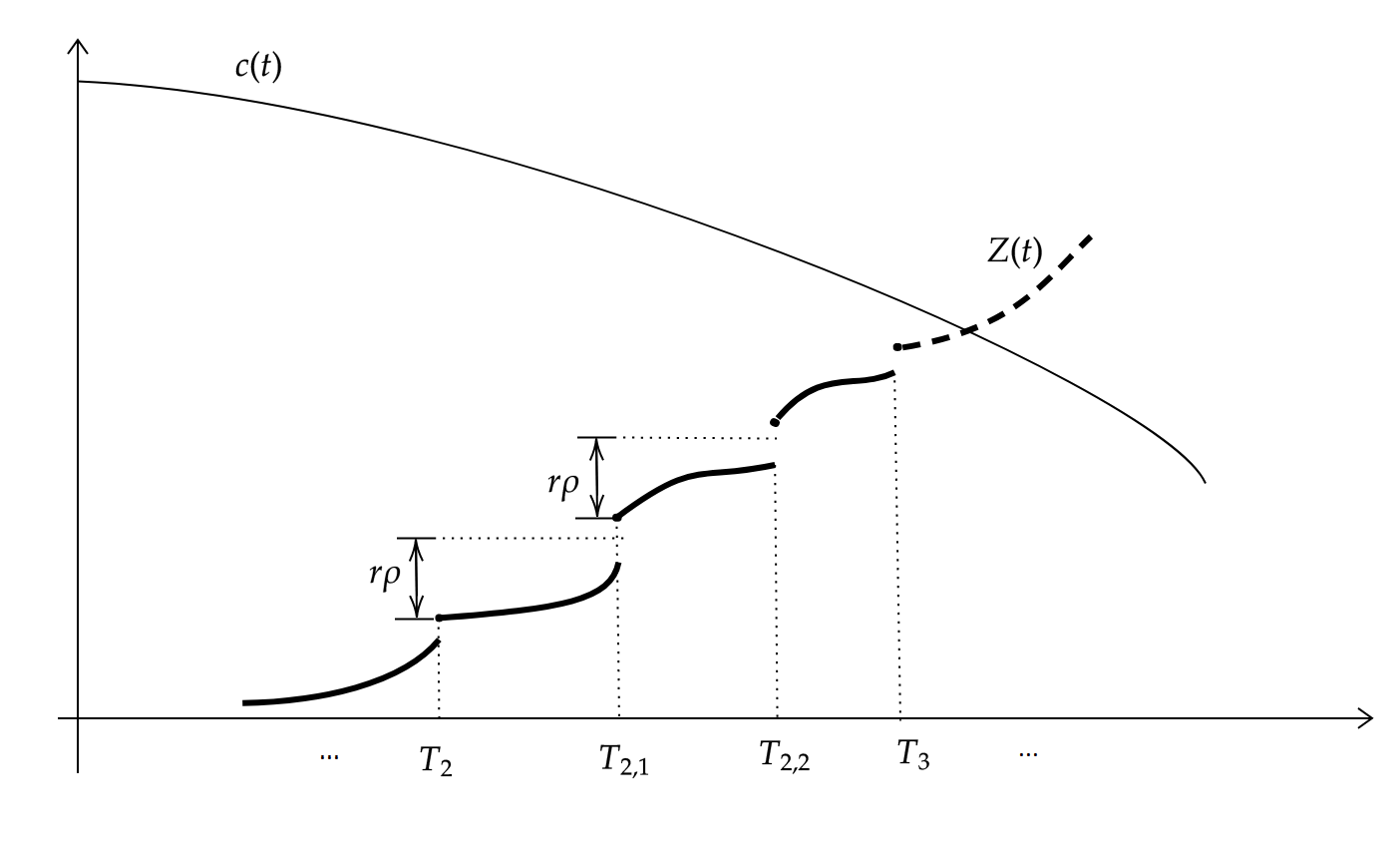}
\caption{Sampling $\chi^Z$ at the times $T_i$ and $T_{i,j}$ }
\label{fig:consequtive_crossing}
\end{figure}


The computational complexity of Algorithm~\ref{alg:triplet_Z} can now be understood as follows. Let $0\eqqcolon \varrho_0<\varrho_1<\ldots<\varrho_M$ be the sequence of stopping times (with respect to the filtration $(\sigma(Y_s,Q_s;s\in[0,t]))_{t\ges 0}$) in~\eqref{eq:number_intervals}. There exist random variables $(G_{n,i})_{n\in\N,i\in\{1,2,3\}}$ and a constant $\kappa_0$ (accounting for deterministic operations and the simulation of elementary distributions), satisfying the following: 
\begin{itemize}[leftmargin=2em, nosep]
    \item $G_{n,1}$ is the random computational cost of sampling from the law of the first passage event of the truncated tempered stable process, with expectation bounded by $\mathcal{T}_0(\alpha,\vartheta,q,\min\{c_0,r\rho\})/\left(1-\rho^\alpha\right)$, where $\mathcal{T}_0$ is the time complexity of~\nameref{alg:FPTS} (the bound on the expected running time of Algorithm~\ref{alg:triplet_Y} is given in Lemma~\ref{lem:local_algs_complexity}(a));
    \item $G_{n,2}$ is the random computational cost of sampling from $\lambda_r/\Lambda_r$ (assume finite moments of all orders); 
    \item $G_{n,3}$ is the random computational cost of sampling from truncated tempered stable variable conditional to be small, with expectation bounded above by a multiple of $5\me^{q\min\{c_0,r\rho\}}$, see Lemma~\ref{lem:local_algs_complexity}(b);
     \item for any $n\in\{1,\ldots,M\}$, the running time of Algorithm~\ref{alg:triplet_Z} between times $\varrho_{n-1}$ and $\varrho_n$ is bounded above by $R_n\coloneqq\kappa_0(1+G_{n,1}+G_{n,2}+G_{n,3})$;
     \item the number of intervals between consecutive distinct times in the sequence in~\eqref{eq:number_intervals}  is bounded above by $K+\lceil c_0/(r\rho)\rceil$, where $K$ is the number of jump times of $Q$ in $[0,\tau_c^Z]$.
\end{itemize}

Note that the total complexity of the algorithm is bounded by $\sum_{n=1}^M R_n$. For any $n\in\N$,  let $\mathcal{G}_n$ be the $\sigma$-algebra generated by $(Y,Q)$ up to time $\varrho_n$. Further note that $M=K+\sum_{n=1}^KL_n$ is a stopping time with respect to $(\mathcal{G}_n)_{n\in\N}$, since $\{M\les n\}=\{Z_{\varrho_n}>c(\varrho_n)\}$. Thus, for $(R_n)_{n\in\N}$ and $M$ to satisfy the assumptions of Lemma~\ref{lem:comp_cost}, it suffices to verify that their corresponding moments are appropriately bounded. 

The sequence $(R_n)_{n\in\N}$ satisfies $\sup_{n\in\N}\E[R_n^p]<\infty$ for all $p\ges 1$ and, by Lemma~\ref{lem:local_algs_complexity}, we have
\begin{equation}
\label{eq:R_n-mean}
\E[R_n|\mathcal{G}_{n-1}]\les C\coloneqq \mathcal{T}_0(\alpha,\vartheta,q,\min\{c_0,r\rho\})/\left(1-\rho^\alpha\right)+5\me^{q\min\{c_0,r\rho\}}+\E[G_{n,2}]\enskip\text{a.s.,}
\quad n\in\N.
\end{equation}
Indeed, $(G_{n,1})_{n\in\N}$ has exponential moments that are uniformly bounded (in $n$) as a function of $c_0$, $r\rho$, $\alpha$, $\vartheta$ and $q$. The sequence $(G_{n,2})_{n\in\N}$ is simply iid with assumed finite moments of all orders (and hence $\E[G_{n,2}^p]=\E[G_{1,2}^p]<\infty$ for all $n\in\N$, $p\ges 1$) and the sequence $(G_{n,3})_{n\in\N}$ are all geometric variables with uniformly bounded (in $n$) polynomial moments which are smaller than the corresponding moment of a geometric random variable with acceptance probability $\me^{-q\min\{c_0,r\rho\}}/5$.
Our next task is to establish a bound on 
$\E[M]$ and  prove that $M$ has moments of all order.
This will, by Lemma~\ref{lem:comp_cost}, imply that the running time of Algorithm~\ref{alg:triplet_Z} has finite moments of any order and, together with~\eqref{eq:R_n-mean}, yield
a bound on the expected running time $\E\left[\sum_{n=1}^M R_n\right]$.

Recall $K+\lceil c_0/(r\rho)\rceil\ges M$,
where $K$ is the number of jumps of $Q$ in the interval $[0,\tau_c^Z]$.
Since $[0,\tau_c^Z]\subset[0,\tau_c^{Y}]$, $K$ is bounded by the number of jumps of $Q$ on the interval $[0,\tau_c^{Y}]$, which is, conditionally given $\tau_c^{Y}$, Poisson distributed with mean $\tau_c^{Y}\Lambda_r$. Hence, $K$ possesses finite exponential moments of any order since, by~\cite[Lem.~A.2]{cazares2023fast}, $\tau_c^Y$ has finite exponential moments of any order. Moreover, by Lemma~\ref{lem:tau}(a) (with $u=1/(1+c_0)$), we have
\begin{equation}
    \label{eq:alg_bound}
\E[K]\les
\E[\tau_c^Y]\Lambda_r
\les \frac{\me^{c_0/(1+c_0)}\Lambda_r}{\vartheta\int_0^{r}(1-\me^{-x/(1+c_0)})\me^{-qx}x^{-\alpha-1}\md x}.
\end{equation}

On the other hand, since $Z=Y+Q$ we have $[0,\tau_c^Z]\subset[0,\tau_c^{Q}]$. Thus, $K$ is also bounded by the number of jumps $Q$ requires to cross the level $c_0$. This number equals the number of steps the jump-chain random walk, embedded in $Q$, requires to cross level $c_0$. Since the increment distribution of this walk is given by $\lambda_r(\md x)/\Lambda_r$, 
where $\lambda_{r}$ is the L\'evy measure in~\eqref{eq:Levy_measure_description},
Lemma~\ref{lem:tau}(b) implies
\begin{equation}
\label{eq:walk_bound}
\begin{split}
\E[K]
&\les\frac{\me^{c_0/(1+c_0)}}{1-\Lambda_r^{-1}\int_0^\infty \me^{-x/(1+c_0)}\lambda_r(\md x)}
=\frac{\me^{c_0/(1+c_0)}\Lambda_r}{\int_0^\infty (1-\me^{-x/(1+c_0)})\lambda_r(\md x)}\\
&=\frac{\me^{c_0/(1+c_0)}\Lambda_r}{\psi_0+\vartheta\int_r^{r_0}(1-\me^{-x/(1+c_0)})\me^{-qx}x^{-\alpha-1}\md x},
\end{split}
\end{equation}
where we recall that $\psi_0=\int_{(0,\infty)}(1-\me^{-x/(1+c_0)})\lambda_{r_0} (\md x)$. Since $\min\{1/x,1/y\}\les 2/(x+y)$ for any $x,y>0$, the bounds in~\eqref{eq:alg_bound} and~\eqref{eq:walk_bound} yield 
\[
\E[K]\les\frac{2\me^{c_0/(1+c_0)}\Lambda_r}{\psi_0+\vartheta\int_0^{r_0}(1-\me^{-x/(1+c_0)})\me^{-qx}x^{-\alpha-1}\md x}\les \frac{2\me\Lambda_r}{\psi_0+\vartheta\Upsilon(1/(1+c_0),r_0,q)},
\]
where the second inequality follows from
Lemma~\ref{lem:bound_psi_u}.

The mean of $\sum_{n=1}^M R_n$ (and thus of the expected running time of Algorithm~\ref{alg:triplet_Z}) is by Lemma~\ref{lem:comp_cost} and~\eqref{eq:R_n-mean} bounded above by a constant multiple of
\begin{equation*}
\left(\mathcal{T}_0(\alpha,\vartheta,q,\min\{c_0,r\rho\})/\left(1-\rho^\alpha\right)+5\me^{q\min\{c_0,r\rho\}}\right)
\left(\frac{2\me\Lambda_r}{\psi_0+\vartheta\Upsilon(1/(1+c_0),r_0,q)}+\frac{c_0}{r\rho}\right).
\end{equation*}
This readily yields~\eqref{eq:main_complexity}.
\end{proof}

\begin{remark}
    In the proof of Theorem~\ref{thm:main_complexity} above, in order to obtain the correct asymptotic bound on the complexity of Algorithm~\ref{alg:triplet_Z}, it was crucial to obtain both bounds~\eqref{eq:alg_bound} and~\eqref{eq:walk_bound}  on the expectation $\E[K]$. This is because, in the regime $\alpha\to0$,
    the truncation level $r$ (which in this case equals $2\alpha/q$) also goes to zero, making the denominator in the bound in~\eqref{eq:alg_bound} go to zero. This makes the bound in~\eqref{eq:alg_bound} deteriorate much faster than the actual behaviour of the algorithm. The denominator of the bound in~\eqref{eq:walk_bound} does not go to zero in this regime. 
\end{remark}

\section*{Acknowledgements}

JGC and AM are supported by the EPSRC grant  EP/V009478/1 and by The Alan Turing
Institute under the EPSRC grant EP/X03870X/1.
AM is also supported by the EPSRC grant EP/W006227/1.
FL is funded by The China Scholarship Council and The University of Warwick full PhD scholarship.

\appendix
\section{How general is condition~\eqref{eq:Levy_measure_description}?}
\label{sec:class_of_subordinators}
The main Algorithm~\ref{alg:triplet_Z} can be applied to quite general L\'evy process. In fact, we have following lemma. We say that $f(t)=\Oh(g(t))$ as $t\downarrow 0$ for a positive function $g$ if $\limsup_{t\downarrow 0}f(t)/g(t)<\infty$. 

\begin{lem}\label{lem:decompose}
Given two L\'evy densities $\varpi_1$ and $\varpi_2$ satisfying $\varpi_2(t)=(1+\Oh(t))\varpi_1(t)$ as $t\downarrow 0$, there exist constants $b,r>0$ and a finite measure $\bar\xi$ on $(0,\infty)$ (i.e. $\bar\xi(0,\infty)<\infty$) such that
\begin{equation}\label{eq:decompose}
    \varpi_2(t)=\me^{-bt}\varpi_1(t)\1\{t\les r\}+\bar\xi(t)\qquad\text{ for }t>0.
\end{equation}
\end{lem}

It is easy to see that, for any driftless subordinator $\bar Z$ with L\'evy density $\bar\varpi$ satisfying $\bar\varpi(t)=(1+\Oh(t))t^{-\alpha-1}$ as $t\downarrow 0$, it can be decomposed into $\bar Z=\bar S+\bar Q$, where $\bar S$ and $\bar Q$ are independent with L\'evy densities
$\me^{-bt}t^{-\alpha-1}\1\{t\les r\}$ and $\bar\xi(t)$, respectively, for some constants $b,r>0$. Also, as a result of Theorem~\ref{thm:main_complexity}, we can optimise the running time by choosing proper $b$ and $\bar \xi$.

\begin{proof}[{\textbf{{Proof of Lemma~\ref{lem:decompose}}}}]

There exists positive constants $a_1, a_2$ and $r_0$ such that 
   $$
   (1-a_1t)\varpi_1(t)\les \varpi_2(t)\les(1+a_2t)\varpi_1(t)$$
    for $0<t\les r_0$. Fix $b\ges a_1$ and $0<r\les r_0$ such that 
$e^{-bt}\les 1-a_1t$
    for $0<t\les r$. If $a_1=0$, then let $b=0$ and $r=r_0$. Define $\phi(t):=e^{-bt}\varpi_1(t)1\{t\les r\}$. It is clear that $\phi(t)\les\varpi_2(t)1\{t\les r\}$. Let $\bar\xi(t):=\varpi_2(t)-\phi(t)$. Note that for $0<t<r$
   \begin{equation*}
   \begin{aligned}
   0\les \bar\xi(t)=\varpi_2(t)-\phi(t)
   &\les (1+a_2t)\varpi_1(t)-\phi(t)\\
   &=(1+a_2t-e^{-bt})\varpi_1(t)
   \les(a_2+b)t\varpi_1(t),
   \end{aligned}
   \end{equation*}
   hence $\bar\xi$ is the Levy measure of some compound Poisson random variable. 
\end{proof}

\section{First moment of first-passage times}
\label{sec:hitting_times}

\begin{lem}\label{lem:tau}
{\normalfont(a)} Let $X$ be a driftless non-decreasing L\'evy process with L\'evy measure $\nu$ on $(0,\infty)$. For any $c_0>0$, define $\tau_{c_0}^X:=\inf\{t>0:X_t>c_0\}$ and $\psi_X(u):=\int_{(0,\infty)}(1-\me^{-u x})\nu(\md x)$ for $u>0$. Then $\E[\tau_{c_0}^X]\les\me^{uc_0}/\psi_X(u)$ for any $u>0$.\\
{\normalfont(b)} Let $R=(R_n)_{n\in\N}$ be a non-decreasing random walk with increment distribution $\mu$ on $(0,\infty)$. For any $c_0>0$, define $\tau_{c_0}^R:=\inf\{n>0:R_n>c_0\}$ and $\psi_R(u):=\E[\me^{-uR_1}]=\int_{(0,\infty)}\me^{-u x}\mu(\md x)$ for $u>0$. Then we have $\E[\tau_{c_0}^R]\les\me^{uc_0}/(1-\psi_R(u))$ for any $u>0$. 
\end{lem}

\begin{proof}[{\textbf{{Proof of Lemma~\ref{lem:tau}}}}]
Part (a) is proved in~\cite[Lem.~A.2]{cazares2023fast}. Part~(b) follows along similar lines: for $u>0$ and $n\in\N$, we have
\[
\p[\tau_{c_0}^R>n]
\les\p[R_n\les c_0]
=\p[\me^{-uR_n}\ges \me^{-uc_0}]
\les \E[\me^{-uR_n}]\me^{u c_0} = \psi_R(u)^n\me^{u c_0}.
\]
Note $\psi_R(u)\in(0,1)$. Thus, 
\[
\E[\tau_{c_0}^R]=
\sum_{n\in\N}n\p[\tau_{c_0}^R=n]=\sum_{n\in\N}\sum_{i=1}^{n}\p[\tau_{c_0}^R=n] = 
\sum_{i\in\N}\sum_{n=i}^\infty \p[\tau_{c_0}^R=i]=\sum_{i\in\N}\p[\tau_{c_0}^R>i-1]\les \sum_{i\in\N}\psi_R(u)^{i-1}\me^{u c_0},
\]
and the bound in part~(b) follows.
\end{proof}

\section{Infinite expected running time of Step~5 of the algorithm  in~\cite[Sec.~7]{chi2016exact}}
\label{app:Chi}

Suppose $Y$, $P$ and $\zeta$ are as in Subsection~\ref{subsec:infinite_complexity_Chi} above. First we show that sampling the undershoot $\zeta_{\tau_{\bar b}^\zeta-}$ as in~\cite{chi2016exact} has infinite expected running time (this proof, originally presented in~\cite[Sec.~2.4]{cazares2023fast}, is recalled here because it facilitates the main proof of this section). Suppose for simplicity that $\theta=1$ and $c\equiv \bar{b}\equiv r$. Consider Step 3 of the algorithm in~\cite[Sec.~7]{chi2016exact}, which uses rejection-sampling to simulate the pair $(\zeta_{\tau_c^\zeta-},\Delta_{\zeta}(\tau_c^\zeta))$ given $\{\tau_c^\zeta=\tau\}$. Recall from~\cite[Prop.~4.1]{cazares2023fast} that $\eta\coloneqq \tau^{-1/\alpha}r=(\tau_c^\zeta)^{-1/\alpha}r\eqd \zeta_1$ follows a stable law. The algorithm proposes $\zeta_{\tau-}=Br$, where $B\sim \mathrm{Beta}(1,1-\alpha)$, with acceptance probability $p=h(\tau^{-1/\alpha}\zeta_{\tau-},U)/M_\alpha=h(B\eta,U)/M_\alpha$. Here $U\sim\Unif(0,1)$ is an auxiliary variable, $h(x,u)\coloneqq \sigma_\alpha(u)x^{-\beta-1}\me^{-\sigma_\alpha(u)x^{-\beta}}$ with $\beta\coloneqq\alpha/(1-\alpha)$ and $\sigma_\alpha$ as in~\eqref{eq:zolotarev_density} above, and $M_\alpha\coloneqq(1-\alpha)^{1-1/\alpha}\alpha^{-1-1/\alpha}\me^{-1/\alpha}$ is a global bound on $h$. Thus, the expected running time of this step equals $\E[\mathfrak{C}]$, where $\mathfrak{C}\coloneqq 1/p\eqd M_\alpha/h(B \eta,U)$ and the variables $U$, $B$ and $\eta$ are independent. Hence, 
\begin{equation}
\label{eq:chi_complexity}
\mathfrak{C}\eqd M_\alpha/h(B \eta,U)=M_\alpha\sigma_\alpha(U)^{-1}(B\eta)^{\beta+1}\exp(\sigma_\alpha(U)(B\eta)^{-\beta})
\ges M_\alpha\sigma_\alpha(U)^{-1}B^{\beta+1}\eta^{\beta+1}.
\end{equation}
However, by e.g.~\cite[Eq.~(4.26)]{cazares2023fast}, $\E[\eta^\ell]=\E[\zeta_1^\ell]<\infty$ if and only if $\ell<\alpha$. Since $1+\beta=1+\alpha/(1-\alpha)>1$ for all $\alpha\in(0,1)$, the expected running time $\E[\mathfrak{C}]$ of this step is indeed infinite. 

Next we show that sampling $Y_{\tau-}$, conditional on $Y_{\tau-}+P_{\tau-}=s\le r$, given in Step~5 of the algorithm  in~\cite[Sec.~7]{chi2016exact}, also has infinite expected running time. Let all variables and functions be as in the previous paragraph. The algorithm in~\cite[Sec.~7]{chi2016exact} first samples a stable variable $s$, conditioned on $\{s\les r\}$, and a discrete variable $\kappa$ with law $\p(\kappa=k)=C_k/C$, where $C_k=[s^{1-\alpha}r^\alpha\eta^{-\alpha}\alpha q]^k/[k!\Gamma(1+(1-\alpha)k)]$ for $k\in\N\cup\{ 0\}$ and \[
C=\sum_{k=0}^\infty C_k
\les 2\sum_{n=0}^\infty\frac{[s^{1-\alpha}r^\alpha\eta^{-\alpha}\alpha q]^k}{k!}
=2\me^{s^{1-\alpha}r^\alpha\eta^{-\alpha}\alpha q}.
\]
(recall that $\eta=\tau^{-1/\alpha}r\eqd \zeta_1$). The simulation of $\kappa$ may itself have a large complexity, but we assume here that it has constant cost. Then the algorithm samples $B'\sim\mathrm{Beta}(1,(1-\alpha)\kappa)$ (if $\kappa=0$, then $B'=1$ a.s.) and $U\sim\Unif(0,1)$ and proposes $Y_{\tau-}=B's$ with acceptance probability bounded below by $p'=h(\eta B's/r ,U)/M_\alpha$. Thus, by~\eqref{eq:chi_complexity} and Jensen's inequality, the expected complexity, conditionally given $s$, is bounded below by
\[
\E[1/p'|s]
\ges M_\alpha(s/r)^{\beta+1}
    \E[\sigma_\alpha(U)^{-1}]
    \E[B'\eta|s]^{\beta+1},
\]
which we claim to be infinite. Note that, a.s., $\E[B'|\kappa,\eta,s]=(1+(1-\alpha)\kappa)^{-1}$, $\p[\kappa=0|\eta,s]=1/C$ and hence 
\[
\E[B'\eta|s]
=\E[(1+(1-\alpha)\kappa)^{-1}\eta|s]
\ges\E[\eta\1_{\{\kappa=0\}}|s]
=\E[\eta/C|s]
\ges \tfrac{1}{2}\me^{-s^{1-\alpha}r^\alpha\alpha q}\E[\eta\1_{\{\eta>1\}}].
\]
Since $\E[\eta]=\E[\zeta_1]=\infty$, the right-most expectation in the display is infinite, implying that 
Step~5 of the algorithm  in~\cite[Sec.~7]{chi2016exact}
 also has infinite expected running time.

\bibliographystyle{plain}

\bibliography{ref_first_passage}

\end{document}